\documentclass[10pt]{amsart}
\usepackage{amsmath,amsfonts,amssymb}
\usepackage{graphicx, verbatim}
\usepackage[numbers,square,sort&compress]{natbib} 
\usepackage{times}
\usepackage{xspace}
\usepackage{xypic}
\usepackage{color}
\usepackage{a4wide}
\usepackage{url}
\usepackage{tikz}
\usetikzlibrary{arrows}
\usetikzlibrary{decorations.markings}
\usetikzlibrary{decorations.pathreplacing}
\usetikzlibrary{patterns}
\usepackage{tkz-graph}
\usetikzlibrary{shapes.geometric}
\usepackage{subfig}
\usepackage{enumerate}

 \newcounter{comments}
\graphicspath{{./}{./figs/}{../../figs/}}

\theoremstyle{definition}
\newtheorem{running}{Running Example}

\newenvironment{runningcont}
{\addtocounter{running}{-1}\begin{running}{\textit{\textbf{{(continued) }}}}}
  {\end{running}}

\theoremstyle{definition}
\newtheorem{defn}{Definition}[section]
\newtheorem{prop}[defn]{Proposition}
\newtheorem{lem}[defn]{Lemma}
\newtheorem{thm}[defn]{Theorem}

\newtheorem{ex}[defn]{Example}
\newtheorem{defn-thm}[defn]{Definition / Theorem}
\newtheorem{conj}[defn]{Conjecture}
\numberwithin{equation}{section}

\def\N{\mathbb{N}}

\def\R{\mathbb{R}}
\def\Z{\mathbb{Z}}

\def\AAA{\mathcal{A}}
\def\HHH{\mathcal{H}}
\def\LLL{\mathcal{L}}

\def\RRR{\mathcal{R}}
\def\TTT{\mathcal{T}}
\def\CCC{\mathcal{C}}

\DeclareMathOperator{\Acyc}{Acyc}
\DeclareMathOperator{\Conj}{Conj}
\DeclareMathOperator{\C}{C}
\DeclareMathOperator{\FC}{FC}
\DeclareMathOperator{\CFC}{CFC}
\DeclareMathOperator{\TFC}{TFC}

\DeclareMathOperator{\cl}{cl}

\DeclareMathOperator{\tor}{tor}
\DeclareMathOperator{\Cyc}{Cyc}

\DeclareMathOperator{\Hasse}{Hasse}
\DeclareMathOperator{\torHasse}{torHasse}
\DeclareMathOperator{\Chambers}{Cham}
\DeclareMathOperator{\Cent}{Cent}

\newcommand{\catname}[1]{{\normalfont\textbf{#1}}}
\newcommand{\HEAP}{\catname{Heap}}
\newcommand{\TORHEAP}{\catname{torHeap}}

\def\<{\langle}
\def\>{\rangle}

\def\longto{\longrightarrow}

\newcommand{\supp}{\mathrm{supp}}
\newcommand{\setmin}{{-}}

\def\wt{\widetilde}

\def\epsilon{\varepsilon}
\newcommand{\COMMENT}[1]{}
\def\c{\mathsf{c}}

\def\u{\mathsf{u}}

\def\w{\mathsf{w}}
\def\x{\mathbf{x}}
\def\y{\mathbf{y}}

\tikzset{bigarrow/.style={decoration={markings,mark=at position 0.86
      with {\arrow[scale=1.25]{>}}}, postaction={decorate}, 
    shorten >= 3pt, shorten <= 3pt}}

\tikzset{bigarrow2/.style={decoration={markings,mark=at position 0.94
      with {\arrow[scale=1.25]{>}}}, postaction={decorate}, 
    shorten >= 3pt, shorten <= 3pt}}

\tikzset{big arrow/.style={decoration={markings,mark=at position 0.93
      with {\arrow[scale=2]{>}}}, postaction={decorate}, shorten >=8pt}}

\tikzset{big dashed arrow/.style={decoration={markings,mark=at position 0.93
      with {\arrow[scale=2]{>}}}, postaction={decorate}, shorten >=8pt, dashed}}

\tikzset{dashed line/.style={decoration={markings,mark=at position 0.93
      with {\arrow[scale=2]{>}}}, postaction={decorate}, shorten >=8pt, dashed}}

\tikzset{ba2/.style={decoration={markings,mark=at position 0.88
      with {\arrow[scale=1.25]{>}}}, postaction={decorate}, 
    shorten >= 5pt, shorten <= 5pt}}

\tikzset{ba3/.style={decoration={markings,mark=at position 0.98
      with {\arrow[scale=1.25]{>}}}, postaction={decorate}, 
    shorten >= 0pt, shorten <= 0pt}}

 \colorlet{lightred}{red!30!white}
 \colorlet{lightblue}{blue!30!white}
 \colorlet{lightyellow}{yellow!30!white}
 \colorlet{keylime}{green!10!white}

\begin{document}
%%==============

\title{Toric Heaps, Cyclic Reducibility, and Conjugacy in Coxeter Groups}

\author[S.~Chao]{Shih-Wei Chao} \address{Department of Mathematics, University of North Georgia, Dahlonega, GA 30597}
\email{shih-wei.chao@ung.edu}

\author[M.~Macauley]{Matthew Macauley} \address{School of Mathematical
  and Statistical Sciences, Clemson University, Clemson, SC 29634}
\email{macaule@clemson.edu}

\thanks{The second author was partially supported by a Simons Foundation Collaboration Grant for Mathematicians, Award \#358242.}

%% 20F55 Group theory and generalizations
%%          Special aspects of infinite or finite groups
%%             Reflection and Coxeter groups
%% 06A06 Order, lattices, ordered algebraic structures 
%%          Ordered sets
%%             Partial order, general
%% 06A75 Order, lattices, ordered algebraic structures
%%          Ordered sets
%%             Generalizations of ordered sets 

\medskip

\begin{abstract}
  As a visualization of Cartier and Foata's ``partially commutative monoid'' theory, G.X.~Viennot introduced \emph{heaps of pieces} in 1986. These are essentially labeled posets satisfying a few additional properties. They~naturally~arise~as~ models~of~reduced~words~in~Coxeter groups. In~this~paper,~we~introduce~a~cyclic~ version, motivated~by~the~idea~of~taking~a~heap~and~wrapping~it into~a~cylinder. We~call~this~object~a~\emph{toric~heap},~as~we formalize it as a labeled toric poset, which is a cyclic version of an ordinary poset. To~define~the~concept~of a~toric~extension,~we~develop a morphism in the category of toric heaps. We study~toric~heaps~in~Coxeter~theory, in~view~of~the~fact~that~a~cyclic shift of a reduced word is simply a conjugate by an initial or terminal generator. This allows us to formalize and study a framework of \emph{cyclic reducibility} in Coxeter theory, and apply it to model~conjugacy.~We~introduce~the notion of \emph{torically reduced}, which is stronger than being cyclically reduced for group elements. This gives rise to a new class of elements called \emph{torically fully commutative} (TFC), which are those that have a unique cyclic commutativity class, and comprise a strictly bigger class than the \emph{cyclically fully commutative} (CFC) elements. We prove several cyclic analogues of results on fully commutative (FC) elements due to Stembridge. We conclude with how this framework fits into recent work in Coxeter groups, and we correct a minor flaw in a few recently published theorems.
\end{abstract}

\maketitle

%%============================================================
\section{Introduction and Overview}\label{sec:intro}
%%============================================================

In mathematics and computer science, a \emph{trace} is a set of
strings or words over an alphabet $S$ for which certain pairs are
allowed to commute. The commutativity rules can be encoded by an
undirected graph $\Gamma$ called the \emph{dependency graph}, where
the vertices are the letters and edges correspond to non-commuting
pairs. Given a trace, the associated \emph{trace monoid} is the set of
finite words under the equivalence relation generated by these
commutations, where the binary operation is concatenation. The
combinatorics of trace monoids were studied by Cartier and Foata in
the 1960s, who called them \emph{partially commutative
  monoids}~\cite{cartier1969problemes}. They are now sometimes known
as \emph{Cartier--Foata monoids}. We will stick with the term ``trace
monoid'' for brevity. In 1986,
G.~X.~Viennot~\cite{viennot1986combinatoire} introduced the theory of
\emph{heaps of pieces}, which is a combinatorial interpretation of
these objects that leads to a nice way to visualize them. The
``pieces'' represent the distinct letters in the alphabet, and a
string is represented by a vertical stack, or ``heap'' of these
pieces. Two pieces overlap vertically if the corresponding letters do
not commute, as elements in the monoid.  A simple example of this
follows.

 \begin{figure}[!ht]
  \centering
  \begin{tikzpicture}[scale=.95, >=stealth']
    %% Balls
    \begin{scope}[shift={(6.5,0)}]
      \draw {(0,0.15) circle (0.15in) node[]{$a$}};
      \draw {(0,1.45) circle (0.15in) node[]{$a$}};
      \draw {(0.4,0.8) circle (0.15in) node[]{$b$}};
      \draw {(0.4,2.1) circle (0.15in) node[]{$b$}};
      \draw {(0.8,0.15) circle (0.15in) node[]{$c$}};
      \draw {(1.2,0.8) circle (0.15in)node[]{$d$} };
      \draw {(1.6,1.45) circle (0.15in) node[]{$e$}};
      %%%
      \draw {(2.8,0.8) circle (0.15in) node[]{$f$}};
      \draw {(3.2,0.15) circle (0.15in) node[]{$g$}};
      \draw {[-] (-0.8,-.22)--(4,-.22)};
    \end{scope}
    %% Heap
    \begin{scope}[shift={(11.5,0)}, shorten >= 3pt, shorten <= 3pt]
      \draw[fill=black] {(0,0) circle (1.5pt) node[label=left:$a\!\!$]{}};
      \draw[fill=black] {(0,1.2) circle (1.5pt) node[label=left:$a\!\!$]{}};
      \draw[fill=black] {(0.6,0.6) circle (1.5pt) node[label=above:$b$]{}};
      \draw[fill=black] {(0.6,1.8) circle (1.5pt) node[label=above:$b$]{}};
      \draw[fill=black] {(1.2,0) circle (1.5pt) node[label=above:$c$]{}};
      \draw[fill=black] {(1.8,0.6) circle (1.5pt) node[label=above:$d$]{}};
      \draw[fill=black] {(2.4,1.2) circle (1.5pt) node[label=above:$e$]{}};
      \draw[fill=black] {(3,0.6) circle (1.5pt) node[label=above:$f$]{}};
      \draw[fill=black] {(3.6,0) circle (1.5pt) node[label=above:$g$]{}};
      \draw[bigarrow] (0,0)--(0.6,0.6);
      \draw[bigarrow] (1.2,0)--(.6,.6); 
      \draw[bigarrow] (.6,.6)--(0,1.2); 
      \draw[bigarrow] (0,1.2)--(.6,1.8);
      \draw[bigarrow] (1.2,0) to (1.8,0.6);
      \draw[bigarrow] (1.8,0.6)--(2.4,1.2);
      \draw[bigarrow] (3.6,0) to (3,0.6);
    \end{scope}
    %% Graph
    \begin{scope}[shift={(0,0)}]
      \draw[fill=black] {(0,0) circle (1.5pt) node[label=above:$a$]{}};
      \draw[fill=black] {(.8,0) circle (1.5pt) node[label=above:$b$]{}};
      \draw[fill=black] {(1.6,0) circle (1.5pt) node[label=above:$c$]{}};
      \draw[fill=black] {(2.4,0) circle (1.5pt) node[label=above:$d$]{}};
      \draw[fill=black] {(3.2,0) circle (1.5pt) node[label=above:$e$]{}};
      \draw[fill=black] {(4,0) circle (1.5pt) node[label=above:$f$]{}};
      \draw[fill=black] {(4.8,0) circle (1.5pt) node[label=above:$g$]{}};
      \draw {[-] (0,0)--(3.2,0) [-] (4,0)--(4.8,0) };
    \end{scope}
  \end{tikzpicture}
  \caption{The dependency graph $\Gamma$ of a trace monoid (left), a heap of pieces (middle), and its associated labeled poset $P$ (right).}
  \label{fig:heap-example1}
\end{figure}
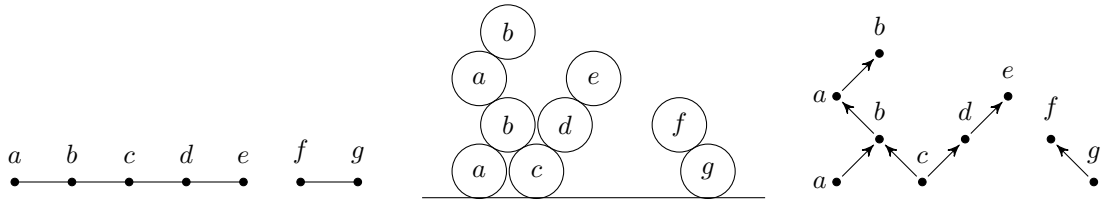

 \begin{ex}\label{ex:heaps-of-pieces}
  Consider the trace monoid $S^*$ over the alphabet $S=\{a,b,c,d,e,f,g\}$ with dependency graph $\Gamma$ shown on the left in Figure~\ref{fig:heap-example1}. That is, two letters commute if and only if they are non-adjacent in $\Gamma$. The string $acbabgdfe$ in $S^*$ defines a heap of pieces shown in the middle of Figure~\ref{fig:heap-example1}. One can think of this as being built by dropping balls in a ``Towers of Hanoi'' fashion onto this dependency graph -- pieces of the same type are aligned vertically, and two pieces of different types overlap vertically if they do not commute. This heap (of pieces) is just a labeled poset (sometimes called its \emph{skeleton}), whose Hasse diagram is shown on the right in Figure~\ref{fig:heap-example1}. Note that every ``labeled'' linear extension is a string that gives rise to the same heap. 

Trace monoids can be defined for arbitrary graphs, though the visualization of the heap in Figure~\ref{fig:heap-example1} works well because $\Gamma$ is a line-graph. For more complicated planar graphs, we might need to make the pieces oddly shaped for the ``Towers of Hanoi'' visualization to work, which originally motivated Viennot's heaps of pieces. For example, in Figure~\ref{fig:heap-example1} the piece $c$ does not commute with either $b$ or $d$. If we want to further require that it does not commute with $e$ and $f$, one way to represent this is to elongate it, as shown in Figure~\ref{fig:heap-example2}. The new dependency graph and Hasse diagram of the labeled poset are shown as well. 

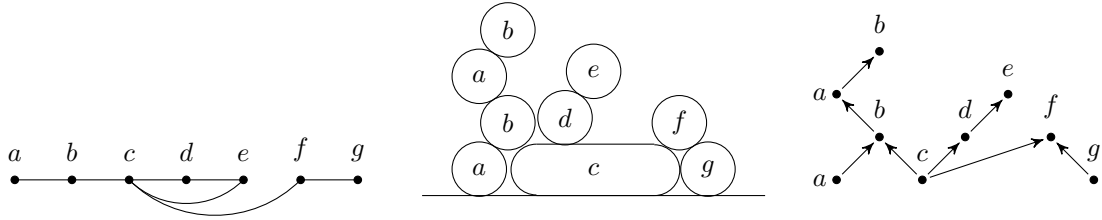
\begin{figure}[!ht]
  \centering
  \begin{tikzpicture}[scale=.95, >=stealth']
    %% Balls
    \begin{scope}[shift={(6.5,0)}]
      \draw {(0,0.15) circle (0.15in) node[]{$a$}};
      \draw {(0,1.45) circle (0.15in) node[]{$a$}};
      \draw {(0.4,0.8) circle (0.15in) node[]{$b$}};
      \draw {(0.4,2.1) circle (0.15in) node[]{$b$}};
      \draw {(0.8,0.5) arc (90:270:0.14in)};
      \draw {(2.45,-0.22) arc (-90:90:0.14in)};
      \draw{[-] (0.8,-0.22)--(2.45,-0.22)};
      \draw{[-] (0.8,0.5)--(2.45,0.5)};
      \draw {(1.2,0.87) circle (0.15in)node[]{$d$} };
      \draw {(1.6,1.52) circle (0.15in) node[]{$e$}};
      \draw {(1.6,0.15) node{$c$}};
      %%%
      \draw {(2.8,0.8) circle (0.15in) node[]{$f$}};
      \draw {(3.2,0.15) circle (0.15in) node[]{$g$}};
      \draw {[-] (-0.8,-.22)--(4,-.22)};
    \end{scope}
      %% Heap
    \begin{scope}[shift={(11.5,0)}, shorten >= 3pt, shorten <= 3pt]
      \draw[fill=black] {(0,0) circle (1.5pt) node[label=left:$a\!\!$]{}};
      \draw[fill=black] {(0,1.2) circle (1.5pt) node[label=left:$a\!\!$]{}};
      \draw[fill=black] {(0.6,0.6) circle (1.5pt) node[label=above:$b$]{}};
      \draw[fill=black] {(0.6,1.8) circle (1.5pt) node[label=above:$b$]{}};
      \draw[fill=black] {(1.2,0) circle (1.5pt) node[label=above:$c$]{}};
      \draw[fill=black] {(1.8,0.6) circle (1.5pt) node[label=above:$d$]{}};
      \draw[fill=black] {(2.4,1.2) circle (1.5pt) node[label=above:$e$]{}};
      \draw[fill=black] {(3,0.6) circle (1.5pt) node[label=above:$f$]{}};
      \draw[fill=black] {(3.6,0) circle (1.5pt) node[label=above:$g$]{}};
      \draw[bigarrow] (0,0)--(0.6,0.6);
      \draw[bigarrow] (1.2,0)--(.6,.6); 
      \draw[bigarrow] (.6,.6)--(0,1.2); 
      \draw[bigarrow] (0,1.2)--(.6,1.8);
      \draw[bigarrow] (1.2,0) to (1.8,0.6);
      \draw[bigarrow] (1.8,0.6)--(2.4,1.2);
      \draw[bigarrow] (3.6,0) to (3,0.6);
      \draw[bigarrow2] (1.2,0) to (3,0.6);
    \end{scope}
    %% Graph
    \begin{scope}[shift={(0,0)}]
      \draw[fill=black] {(0,0) circle (1.5pt) node[label=above:$a$]{}};
      \draw[fill=black] {(.8,0) circle (1.5pt) node[label=above:$b$]{}};
      \draw[fill=black] {(1.6,0) circle (1.5pt) node[label=above:$c$]{}};
      \draw[fill=black] {(2.4,0) circle (1.5pt) node[label=above:$d$]{}};
      \draw[fill=black] {(3.2,0) circle (1.5pt) node[label=above:$e$]{}};
      \draw[fill=black] {(4,0) circle (1.5pt) node[label=above:$f$]{}};
      \draw[fill=black] {(4.8,0) circle (1.5pt) node[label=above:$g$]{}};
      \draw {[-] (0,0)--(3.2,0) [-] (4,0)--(4.8,0) };
      \draw [-] (1.6,0) to [bend right=45] (3.2,0);
      \draw [-] (1.6,0) to [bend right=45] (4,0);
    \end{scope}
  \end{tikzpicture}
  \caption{This heap of pieces is created from the heap
    in Figure~\ref{fig:heap-example1} by additionally restricting
    $c$ from commuting with $e$ and $f$. This adds new edges to the dependency graph and new relations to the poset.}
  \label{fig:heap-example2}
\end{figure}
\end{ex}

For more complicated dependency graphs $\Gamma$, e.g., non-planar ones, the labeled poset arising from a trace monoid over $\Gamma$ does not have a nice visual realization in $2$- or $3$-dimensional space as a stack of pieces. However, there is still an underlying labeled poset which we can rigorously formalize as a heap. In Section~\ref{sec:heaps}, we will formally define all terms so there is no ambiguity about our notation. However, in the remainder of this section, we will assume a few basic definitions that the reader likely already knows, so we can summarize the outline, goals, and main ideas of this paper. 

Since Viennot introduced them in 1986, heaps have been defined in various ways, depending on the context, and usually with the ``of pieces'' dropped from the name. The following definition is due to R.M. Green \cite{green2010combinatorics}, who defined the category of heaps and applied it to Lie theory. Having a category makes definitions like a subheap and a morphism between heaps both natural and precise, and we will revisit this in Section~\ref{sec:heaps}.

\begin{defn}\label{defn:heap}
  A \emph{heap} is a triple $(P,\Gamma,\phi)$ consisting of a poset
  $P$, a graph $\Gamma$, and a function $\phi\colon P\to\Gamma$ to its
  vertex set, satisfying:
  \begin{enumerate}[(i)]
  \item For every vertex $s$ of $\Gamma$, the subset
    $\phi^{-1}(\{s\})$ is a chain in $P$, called a \emph{vertex
    chain}.
  \item For every edge $\{s,t\}$ of $\Gamma$, the subset
    $\phi^{-1}(\{s,t\})$ is a chain in $P$, called an \emph{edge
    chain}.
  \item If $P'$ is another poset over the same set satisfying (i) and (ii), then $P'$ is an extension of $P$.
  \end{enumerate}
\end{defn}

Heaps arise naturally in Coxeter theory, because every reduced word in a Coxeter group can be thought of as a labeled linear extension of a heap over the Coxeter graph $\Gamma$. This is best seen by an example, and the one that follows should be quite illustrative. It will be a running example that we will revisit throughout this paper. 

%\begin{ex}
\begin{running}\label{run:b2}
  Consider the finite Coxeter group $W(B_2)$, whose Coxeter graph is shown in Figure~\ref{fig:b2} on the left.\footnote{Normally, the vertices of $\Gamma(B_2)$ are $s_0,s_1,s_2$, but we are using $s_1,s_2,s_3$, for consistency with the vertex sets of $\Gamma(A_3)$ and $\Gamma(H_3)$, which will appear in later examples.} In this Coxeter group, $s_1s_2s_1s_2=s_2s_1s_2s_1$, and so the element $w=s_3s_1s_2s_1s_2$ can also be written as $w=s_3s_2s_1s_2s_1$. Both of these reduced words gives rise to a heap, which are shown in Figure~\ref{fig:b2}. It is easy to see that these two heaps describe different words in the trace monoid $S^*$, where $S=\{s_1,s_2,s_3\}$. However, they represent the same group element in $W(B_2)$. 
\begin{figure}[!ht]
  \begin{tikzpicture}[shorten >= 3pt, shorten <= 3pt,>=stealth']
    \begin{scope}[scale=.8,shift={(-.5,-.5)}]
      \draw[fill=black] {(0,0) circle (1.5pt) node[label=below:$s_1$]{}};
      \draw[fill=black] {(1.25,0) circle (1.5pt) node[label=below:$s_2$]{}};
      \draw[fill=black] {(2.5,0) circle (1.5pt) node[label=below:$s_3$]{}};
      \draw {(0.625,0) node[label=above:$4$]{}};
      \draw {[-] (0,0)--(1.25,0) [-] }; \draw {[-] (1.25,0)--(2.5,0) [-] };
    \end{scope}
    \begin{scope}[scale=.8,shift={(4,-.5)}]
      \draw {[-] (-.5,0)--(1.5,0) [-] };
      \draw {(0,0.4) circle (0.15in) node[]{$s_1$}};
      \draw {(.4,1.05) circle (0.15in) node[]{$s_2$}};
      \draw {(0,1.7) circle (0.15in) node[]{$s_1$}};
      \draw {(.4,2.35) circle (0.15in) node[]{$s_2$}};
      \draw {(.8,0.4) circle (0.15in) node[]{$s_3$}};
    \end{scope}
    \begin{scope}[scale=.8,shift={(7,0)}]
      \draw[fill=black] {(0,0) circle (1.5pt) node[label=left:$s_1$]{}};
      \draw[fill=black] {(0,1.2) circle (1.5pt) node[label=left:$s_1$]{}};
      \draw[fill=black] {(.6,0.6) circle (1.5pt) node[label=right:$s_2$]{}};
      \draw[fill=black] {(.6,1.8) circle (1.5pt) node[label=right:$s_2$]{}};
      \draw[fill=black] {(1.2,0) circle (1.5pt) node[label=right:$s_3$]{}};
      \draw[bigarrow] (0,0) -- (.6,.6); \draw[bigarrow] (.6,.6) -- (0,1.2);
      \draw[bigarrow] (0,1.2) -- (.6,1.8); \draw[bigarrow] (1.2,0) -- (.6,.6);
    \end{scope}
    \begin{scope}[scale=.8,shift={(11,-.5)}]
      \draw {[-] (-.5,0)--(1.5,0) [-] };
      \draw {(0,3.0) circle (0.15in) node[]{$s_1$}};
      \draw {(.4,1.05) circle (0.15in) node[]{$s_2$}};
      \draw {(0,1.7) circle (0.15in) node[]{$s_1$}};
      \draw {(.4,2.35) circle (0.15in) node[]{$s_2$}};
      \draw {(.8,0.4) circle (0.15in) node[]{$s_3$}};
    \end{scope}
    \begin{scope}[scale=.8,shift={(14,0)},>=stealth']
      \draw[fill=black] {(0,1.2) circle (1.5pt) node[label=left:$s_1$]{}};
      \draw[fill=black] {(.6,0.6) circle (1.5pt) node[label=right:$s_2$]{}};
      \draw[fill=black] {(.6,1.8) circle (1.5pt) node[label=right:$s_2$]{}};
      \draw[fill=black] {(1.2,0) circle (1.5pt) node[label=right:$s_3$]{}};
      \draw[fill=black] {(0,2.4) circle (1.5pt) node[label=left:$s_1$]{}};
      \draw[bigarrow] (.6,1.8) -- (0,2.4); \draw[bigarrow] (.6,.6) -- (0,1.2);
      \draw[bigarrow] (0,1.2) -- (.6,1.8); \draw[bigarrow] (1.2,0) -- (.6,.6);
    \end{scope}
  \end{tikzpicture}
  \caption{The element $w=s_3s_1s_2s_1s_2=s_3s_2s_1s_2s_1$ in the Coxeter group
    $W(B_2)$ has two heaps, one for each commutativity class.}\label{fig:b2}
\end{figure}
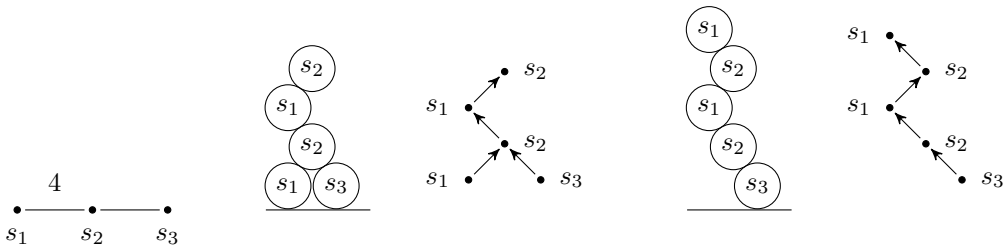
\end{running}

Heaps generally do not provide a ``magic bullet'' for proving theorems in Coxeter theory or elsewhere, but they are often quite useful. They have been applied to a variety of topics in pure and applied mathematics, physics, computer science, and engineering. Examples include fully commutative \cite{stembridge1996fully,stembridge1998enumeration} and freely braided \cite{green2002freely} elements in Coxeter groups, Kazhdan-Lusztig polynomials \cite{billey2001kazhdan}, representations of Kac-Moody \cite{green2007full} and Lie algebras \cite{wildberger2003combinatorial}, $Q$-system cluster algebras \cite{di2010q-systems}, parallelogram polyominoes \cite{bousquet1992empilements}, $q$-analogues of Bessel functions \cite{fedou1995fonctions}, Lyndon words \cite{lalonde1995lyndon}, lattice animals, \cite{bousquet2002lattice}, Motzkin path models for polymers \cite{brak2007motzkin}, Lorentzian quantum gravity \cite{viennot2014heaps}, modeling with Petri nets \cite{gaubert1999modeling}, control theory of discrete-event systems \cite{su2012synthesis}, and many more.

Returning to Coxeter groups, heaps provide a framework for \emph{reducibility}: commutativity classes of elements correspond to heaps, and reduced words to labeled linear extensions. The goal of this paper is to develop and study a cyclic version of a heap. This was originally motivated by a need for a framework of \emph{cyclic reducibility} in Coxeter groups, though we expect that this structure will appear in other settings in combinatorics and beyond. Cyclic reducibility in Coxeter groups is closely related to the conjugacy problem, but it is also interesting in its own right. To motivate this connection, note that conjugating a reduced word $s_{i_1}\cdots s_{i_k}$ by the initial generator $s_{i_1}=s_{i_1}^{-1}$ cyclically shifts it, e.g.,
\begin{equation}\label{eqn:cyclic-shift}
  s_{i_1}(s_{i_1}s_{i_2}\cdots s_{i_k})s_{i_1}=s_{i_2}\cdots s_{i_k}s_{i_1}.
\end{equation}
Loosely speaking, one can think of our cyclic version of a heap as the result of identifying (or gluing) the top with the bottom of the diagrams in Figures~\ref{fig:heap-example1}, \ref{fig:heap-example2} and \ref{fig:b2}, so that the ``heap of pieces'' is not a vertical stack, but rather a cylinder. For simple examples, such as the ones already given, this concept is visually clear. However, it is much less clear to how to formalize this mathematically and what the underlying structure should be, especially for general dependency graphs. 

The answer to this involves a fairly new concept of a \emph{toric poset}, introduced by Develin, Macauley, and Reiner in 2016 \cite{develin2016toric}. A toric poset is a cyclic version of an ordinary poset, that is generated by the equivalence under making minimal elements maximal, in the sense of Eq.~\eqref{eqn:cyclic-shift} above. Many fundamental features of posets have very natural cyclic, or ``toric'' analogues. For example, a chain in a poset is a totally ordered set, but a toric chain in a toric poset represents a totally cyclically ordered set. An extension of a poset is defined by adding relations. The toric counterpart to this concept is called a toric extension, but in order to see how these are analogous, one has to view things geometrically, and that is where the ``toric'' name comes from. A finite poset $P$ can be viewed as an acyclic directed graph (but not uniquely), and that can be associated with a chamber $c(P)$ of a graphic hyperplane arrangement $\AAA(G)$ in $\R^n$. Though $G$ and hence $\AAA(G)$ are not uniquely determined by the poset, the particular chamber $c(P)\subseteq\R^n$ is. The geometric interpretation of the equivalence generated by making minimal elements maximal is quotienting out by the integer lattice $\Z^n$. The result is a (toric) hyperplane arrangement $\AAA_{\tor}(G)$ in the $n$-torus $\R^n\!/\Z^n$. The chambers of $\AAA_{\tor}(G)$ are in bijection with the acyclic orientations of $G$ under the equivalence of converting sources into sinks, and these are called \emph{toric posets} over $G$. Now, back to extensions: an extension of a poset can be described geometrically as adding hyperplanes to the arrangement $\AAA(G)$, and a (toric) extension of a toric poset corresponds to adding (toric) hyperplanes to $\AAA_{\tor}(G)$. There are also natural toric analogues of linear extensions, transitivity, Hasse diagrams, intervals, antichains, order ideals, morphisms, and $P$-partitions, among others. The ones relevant to toric heaps will be discussed later when we formalize them in Section~\ref{subsec:toric-posets}. More details about these and others can be found in \cite{develin2016toric,macauley2016morphisms,adin2018cyclic}. 

The formal definition of a toric poset can be found in Definition / Theorem \ref{defn-thm:toric-poset}. However, now that we have conveyed the intuitive idea of it, we can give the formal definition of a toric heap. It should be thought of as a labeled toric poset -- a cyclic version of an ordinary heap.

\begin{defn}\label{defn:toric-heap}
  A \emph{toric heap} is a triple $(T,\Gamma,\tau)$ consisting of a
  toric poset $T$, a graph $\Gamma$, and a function $\tau\colon
  T\to\Gamma$ to its vertex set, satisfying:
  \begin{enumerate}[(i)]
  \item For every vertex $s$ of $\Gamma$, the subset
    $\tau^{-1}(\{s\})$ is a toric chain in $T$, called a \emph{toric vertex
    chain}.
  \item For every edge $\{s,t\}$ of $\Gamma$, the subset
    $\tau^{-1}(\{s,t\})$ is a toric chain in $T$, called a \emph{toric edge
    chain}.
  \item If $T'$ is another toric poset over the same set satisfying (i) and (ii), then $T'$ is a toric extension of $T$.
  \end{enumerate}
\end{defn}

The remainder of this paper is organized as follows. In Section~\ref{sec:coxeter}, we discuss posets over graphs, and review relevant concepts in Coxeter theory. We use these concepts as motivating examples in Section~\ref{sec:heaps}, where we further study heaps over graphs and the resulting categories. In Section~\ref{sec:coxeter-fc-cfc}, we show how conjugation of Coxeter elements can be described by a source-to-sink operation on acyclic orientations. Generalizing this construction leads to the fully commutative (FC) and cyclically fully commutative (CFC) elements, which illustrate the need for a framework of cyclic reducibility in Coxeter theory. In Section~\ref{sec:toric}, we review the concept of a toric poset, both geometrically as a chamber of a graphic toric hyperplane arrangement, and combinatorially as an equivalence class of acyclic orientations. This allows us to finally define toric heaps formally, and we look at the resulting categories. In Section~\ref{sec:cyclic-reducibility}, we formalize cyclic reducibility in Coxeter groups. We start with cyclic words, and distinguish between words and group elements being cyclically reduced and torically reduced. This leads us to the concept of cyclic commutativity classes. Next, we define the toric heap of a torically reduced word in any Coxeter group. In Section~\ref{sec:faux-cfc}, we analyze new classes of elements that arise in this paper that are called \emph{torically fully commutative (TFC)} and \emph{faux CFC}. The former are the elements that have only one cyclic commutativity class, and the latter are those that additionally admit long braid relations. Finally, in Section~\ref{sec:conjugacy}, we discuss recent work on reducibility, cyclic reducibility, and conjugacy using the toric heap framework. A few beautiful results by T.~Marquis in \cite{marquis2014conjugacy} are inaccurate as stated, but easily corrected by replacing ``cyclically reduced'' with ``torically reduced.'' We conclude in Section~\ref{sec:conclusions} with some open problems and directions for future research.

%%========================================================
\section{Combinatorial Coxeter Theory}\label{sec:coxeter}
%%========================================================

\noindent Though the theory of heaps can be developed independently, Coxeter groups provide a wealth of useful and motivating examples, and so we will introduce them right away. More information can be found in classic texts such as Humphreys \cite{humphreys1990reflection} or Bj\"orner and Brenti  \cite{bjorner2005combinatorics}. In this section, we will begin by defining posets over graphs, and then show how they arise in Coxeter theory. That will naturally lead us into heaps, which will be done in Section~\ref{sec:heaps}.

%%---------------------------------------------------
\subsection{Posets over graphs}\label{subsec:posets}
%%---------------------------------------------------

Throughout this paper, $G=(V,E)$ will be an undirected graph without loops, $P$ a nonempty finite set, and $\leq_P$ a binary relation that is reflexive, antisymmetric, and transitive. The pair $(P,\leq_P)$ is a \emph{partially ordered set}, or \emph{poset}. Usually we will write $P$ instead of $(P,\leq_P)$, as the relation is generally understood. 

An \emph{acyclic orientation} $\omega$ of $G$ determines a partial ordering on $V$, where $i\leq_P j$ if and only if there is an $\omega$-directed path from $i$ to $j$. We denote this poset by $P=P(G,\omega)$, and say that it is a \emph{poset on $G$}. Let $\Acyc(G)$ be the set of all acyclic orientations of $G$. It should be noted that a (finite) poset does not uniquely determine a graph. However, given a poset $P$, there is a unique minimal graph $\hat{G}^{\Hasse}(P)$ with respect to edge-inclusion, called the \emph{Hasse diagram}, and a unique maximal graph $\bar{G}(P)$, called the \emph{transitive closure}, from which $P$ arises as an acyclic orientation. A simple example of this is shown in Figure~\ref{fig:4-digraphs}, where acyclic orientations of four different graphs (the undirected versions of those shown) all describe the same $5$-element poset.

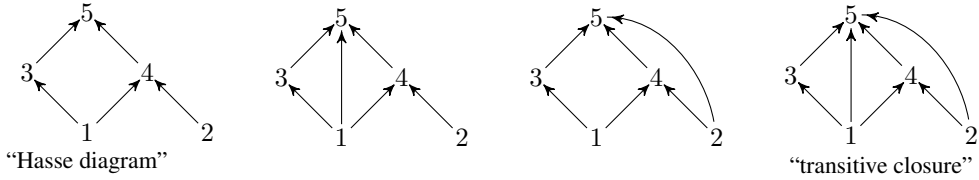
\begin{figure}[!ht]
  \begin{tikzpicture}[scale=0.4, >=stealth', shorten >= 4pt, shorten <= 6pt] 
    { \draw (2,0) node{$1$}; 
      \draw (6,0) node{$2$};
      \draw (0,2) node{$3$}; 
      \draw (4,2) node{$4$};
      \draw (2,4) node{$5$}; 
      \draw[ba2] (2,0) -- (0,2);
      \draw[ba2] (2,0) -- (4,2);
      \draw[ba2] (6,0) -- (4,2);
      \draw[ba2] (0,2) -- (2,4);
      \draw[ba2] (4,2) -- (2,4); }
    \draw (2,-1) node{{\small ``Hasse diagram''}};
  \end{tikzpicture} \hspace{4mm}
  \begin{tikzpicture}[scale=0.4, >=stealth', shorten >= 4pt, shorten <= 6pt] 
    { \draw (2,0) node{$1$}; 
      \draw (6,0) node{$2$};
      \draw (0,2) node{$3$}; 
      \draw (4,2) node{$4$};
      \draw (2,4) node{$5$}; 
      \draw[ba2] (2,0) -- (0,2);
      \draw[ba2] (2,0) -- (4,2);
      \draw[ba2] (6,0) -- (4,2);
      \draw[ba2] (0,2) -- (2,4);
      \draw[ba2] (4,2) -- (2,4); 
      \draw[ba2] (2,0) -- (2,4);
      \draw (2,-1.2) node{\;\;\;};}
  \end{tikzpicture} \hspace{4mm}
  \begin{tikzpicture}[scale=0.4, >=stealth', shorten >= 4pt, shorten <= 6pt] 
    { \draw (2,0) node{$1$}; 
      \draw (6,0) node{$2$};
      \draw (0,2) node{$3$}; 
      \draw (4,2) node{$4$};
      \draw (2,4) node{$5$}; 
      \draw[ba2] (2,0) -- (0,2);
      \draw[ba2] (2,0) -- (4,2);
      \draw[ba2] (6,0) -- (4,2);
      \draw[ba2] (0,2) -- (2,4);
      \draw[ba2] (4,2) -- (2,4); 
      \draw (6,0) edge[out=100,in=0] (2,4);
      \draw[bigarrow] (2.6,3.98) -- (2.5,4); % Stupid hack
      \draw (2,-1.2) node{\;\;\;};}
  \end{tikzpicture} \hspace{4mm}
  \begin{tikzpicture}[scale=0.4, >=stealth', shorten >= 4pt, shorten <= 6pt] 
    { \draw (2,0) node{$1$}; 
      \draw (6,0) node{$2$};
      \draw (0,2) node{$3$}; 
      \draw (4,2) node{$4$};
      \draw (2,4) node{$5$}; 
      \draw[ba2] (2,0) -- (0,2);
      \draw[ba2] (2,0) -- (4,2);
      \draw[ba2] (6,0) -- (4,2);
      \draw[ba2] (0,2) -- (2,4);
      \draw[ba2] (4,2) -- (2,4);
      \draw[ba2] (2,0) -- (2,4); 
      \draw (6,0) edge[out=100,in=0] (2,4);
      \draw[bigarrow] (2.6,3.98) -- (2.5,4); % Stupid hack
      \draw (3,-1) node{{\small ``transitive closure''}};
    }
  \end{tikzpicture} \caption{Four acyclic directed graphs that describe the same $5$-element poset.}\label{fig:4-digraphs}
\end{figure}

If for every $x\neq y$ in $P$, either $x\leq_P y$ or $y\leq_P x$ holds, then $\leq_P$ is a \emph{total order}, and $(P,\leq_P)$ is a \emph{totally ordered set}. Naturally, we write $x<_P y$ if $x\leq_P y$ and $x\neq y$. A totally ordered subset of a poset is called a \emph{chain}. 

%%---------------------------------------------------
\subsection{Coxeter Groups}\label{subsec:coxeter}
%%---------------------------------------------------

A rank-$n$ \emph{Coxeter system} is a pair $(W,S)$ consisting of a set $S=\{s_1,\dots,s_n\}$ that generates a \emph{Coxeter group} $W$ by the presentation
\[
W=\<s_1,\dots,s_n\mid (s_is_j)^{m_{i,j}}=1\>.
\]
Each \emph{bond strength} $m_{i,j}:=m(s_i,s_j)=1$ if and only if
$s_i=s_j$, and $m_{i,j}$ is precisely the order\footnote{For ease of notation, we allow $m_{i,j}=\infty$, and say that $w^\infty:=1$ for any $w\in W$.} of $s_is_j$. Distinct generators $s_i,s_j$ commute if and only if $m(s_i,s_j)=2$. A Coxeter system has a \emph{Coxeter graph} $\Gamma$ which has vertex set $\{1,\dots,n\}$ (or alternatively, $\{s_1,\dots,s_n\}$) and an edge $\{i,j\}$ with label $m(s_i,s_j)$ for each noncommuting pair of generators. Labels of $3$ are usually omitted because they are the most common. A Coxeter system $(W,S)$ is \emph{irreducible} if $\Gamma$ is connected. 

If a word $\w=s_{x_1}\cdots s_{x_m}\in S^*$ is equal to $w$ when considered as an element of $W$, we say that it is a \emph{word} or \emph{expression} for $w$. If furthermore, $m$ is minimal, we call it a \emph{reduced word} for $w$, and we call $m$ its \emph{length}, denoted $\ell(w)$. Let $\RRR(w)$ be the set of reduced words for $w\in W$ and $\RRR(W,S^*)$ be the set of all reduced words. We typically write words using {\sf san serif} font, though it is common to speak of a word $\w\in S^*$ as also being a group element $w\in W$.

For each integer $m\geq 2$ and distinct generators $s,t\in S$, define
\[
\<s,t\>_m=\underbrace{stst\cdots}_{m}\in S^*.
\]
A relation of the form $\<s,t\>_{m(s,t)}=\<t,s\>_{m(s,t)}$ is a
\emph{braid relation}, and a \emph{short braid relation}\footnote{Some authors call $\<s,t\>_{m(s,t)}=\<t,s\>_{m(s,t)}$ a short braid relation if $m(s,t)=3$, and a commutation relation if $m(s,t)=2$.} if $m(s,t)=2$. The braid relations generate an equivalence on $S^*$, denoted $\approx$. A classic theorem of Matsumoto~\cite{matsumoto1964generateurs} says that the resulting equivalence classes are in bijection with the elements of $W$.

\begin{thm}[Matsumoto]
  Any two reduced words for $w\in W$ differ only by braid relations.
\end{thm}

By Matsumoto's theorem, it is well-defined to let the \emph{support}
of an element $w \in W$, denoted $\supp(w)$, be the set of all
generators appearing in any reduced word for $w$. If
$\supp(w)=S$, then we say that $w$ has \emph{full support}. 

The short braid relations generate an equivalence relation $\sim$
on $S^*$ that is coarser than $\approx$. The resulting equivalence classes are called \emph{commutativity classes}. Clearly, the reduced words of any $w\in W$ are a disjoint union of commutativity classes, i.e.,
\[
\RRR(w)=\mathcal{C}(\w_1)\cup\cdots\cup\mathcal{C}(\w_k)
\]
for some reduced words $\w_1,\dots,\w_k$, and where $\mathcal{C}(\w_i)$ is the commutativity class that contains $\w_i$. 

\begin{defn}
  An element $w\in W$ is \emph{fully commutative} (FC) if $\RRR(w)$ contains only one commutativity class. Let $\FC(W,S)$ denote the set of fully commutative elements of $W$.
\end{defn}

The classification of finite and affine Coxeter groups is well known, and it consists of several infinite families and some exceptional cases \cite{bjorner2005combinatorics}. We will denote these groups by e.g., $W(A_n)$, $W(\widetilde{B}_n)$, and their Coxeter graphs by, e.g., $\Gamma(A_n)$, $\Gamma(\widetilde{B}_n)$, etc.

\begin{runningcont}%[continues=run:b2]
  Consider the word $\w=s_3s_1s_2s_1s_2$ as an element of three different Coxeter groups, one for each of the Coxeter graphs shown below. Recall that we are deliberately using $\{s_1,s_2,s_3\}$ instead of the usual $\{s_0,s_1,s_2\}$ as the generating set of $W(B_2)$ so that all three Coxeter graphs have the same vertex sets. 
  
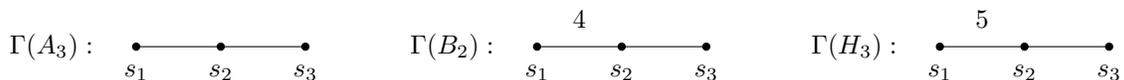
\[
  \begin{tikzpicture}[scale=.9, >=stealth', shorten >= 1pt, shorten <= 1pt]
    \begin{scope}
      \draw {(-.36,0) node[label=left:$\Gamma(A_3):$]{}};
      \draw[fill=black] {(0,0) circle (1.5pt) node[label=below:$s_1$]{}};
      \draw[fill=black] {(1.25,0) circle (1.5pt) node[label=below:$s_2$]{}};
      \draw[fill=black] {(2.5,0) circle (1.5pt) node[label=below:$s_3$]{}};
      \draw {[-] (0,0)--(1.25,0) [-] }; \draw {[-] (1.25,0)--(2.5,0) [-] };
    \end{scope}
  \end{tikzpicture}
  \hspace{10mm}
  \begin{tikzpicture}[scale=.9, >=stealth', shorten >= 1pt, shorten <= 1pt]
    \begin{scope}
      \draw {(-.36,0) node[label=left:$\Gamma(B_2):$]{}};
      \draw[fill=black] {(0,0) circle (1.5pt) node[label=below:$s_1$]{}};
      \draw[fill=black] {(1.25,0) circle (1.5pt) node[label=below:$s_2$]{}};
      \draw[fill=black] {(2.5,0) circle (1.5pt) node[label=below:$s_3$]{}};
      \draw {(0.625,0) node[label=above:$4$]{}};
      \draw {[-] (0,0)--(1.25,0) [-] }; \draw {[-] (1.25,0)--(2.5,0) [-] };
    \end{scope}
  \end{tikzpicture}
  \hspace{10mm}
  \begin{tikzpicture}[scale=.9, >=stealth', shorten >= 1pt, shorten <= 1pt]
    \begin{scope}
      \draw {(-.36,0) node[label=left:$\Gamma(H_3):$]{}};
      \draw[fill=black] {(0,0) circle (1.5pt) node[label=below:$s_1$]{}};
      \draw[fill=black] {(1.25,0) circle (1.5pt) node[label=below:$s_2$]{}};
      \draw[fill=black] {(2.5,0) circle (1.5pt) node[label=below:$s_3$]{}};
      \draw {(0.625,0) node[label=above:$5$]{}};
      \draw {[-] (0,0)--(1.25,0) [-] }; \draw {[-] (1.25,0)--(2.5,0) [-] };
    \end{scope}
\end{tikzpicture}
\]
The word $\w=s_3s_1s_2s_1s_2$ is not reduced in $W(A_3)$ because
$s_3(s_1s_2s_1)s_2=s_3(s_2s_1s_2)s_2=s_3s_2s_1$. It is reduced in $W(B_2)$ but not FC, because $w=s_3(s_1s_2s_1s_2)=s_3(s_2s_1s_2s_1)$. The partition of the reduced words into commutativity classes is
\[
\RRR(w)=\{s_3s_1s_2s_1s_2,\,s_1s_3s_2s_1s_2\}\cup\{s_3s_2s_1s_2s_1\}.
\]
Finally, the word $\w$ is reduced in $W(H_3)$ and the corresponding group element $w$ has a unique commutativity class, $\RRR(w)=\{s_3s_1s_2s_1s_2,s_1s_3s_2s_1s_2\}$, so it is FC.
\end{runningcont}

%%====================================================
\section{Labeled posets and heaps}\label{sec:heaps}
%%====================================================

Recall from Definition~\ref{defn:heap} that a heap is a triple $(P,\Gamma,\phi)$, where $\phi\colon P\to\Gamma$ is a map from a poset to a graph. We will call $P$ the \emph{heap poset}, $\Gamma$ the \emph{heap graph}, and $\phi$ the \emph{labeling map}. Recall that the partial order on $P$ is minimal (coarsest) such that the preimage $\phi^{-1}(s)$ of each vertex and the preimage $\phi^{-1}(\{s,t\})$ of each edge in $\Gamma$ are chains. However, when defining the heap from a concrete object, such as a reduced word in a Coxeter group, it is also necessary to specify the relative order of the elements within each of these chains. We will do this with an acyclic orientation.

\begin{defn} \label{defn:w-graph}
Given a word $\w=s_{x_1}\cdots s_{x_m}$ in $S^*$, consider the graph $G_\w=(V,E)$, where $V=[m]$ and $E$ is the set of all $\{i,j\}$ for which $i\neq j$ and $m(s_{x_i},s_{x_j})\neq 2$. Let $\omega_\w$ be the orientation where each edge $\{i,j\}$ is oriented as $i\to j$ if $i<j$ and $j\to i$ otherwise. Define the poset $P_\w=P(G_\w,\omega_\w)$. 
\end{defn}

\begin{defn} \label{defn:heap-coxeter}
  Fix a Coxeter system $(W,S)$, and let $\w=s_{x_1}\cdots s_{x_m}$ be a word in $S^*$. Define the labeling map 
  \[
  \phi_\w\colon P_\w\longto\Gamma\,,\qquad \phi_\w(i)=s_{x_i}.
  \]
  The triple $\HHH(\w):=(P_\w,\Gamma,\phi_\w)$ is called the \emph{heap of $\w$}. If $\w\in\RRR(w)$, then we say it is a heap of the group element $w\in W$. 
\end{defn}

By construction, distinct words in $S^*$ give rise to distinct heaps, even if they are in the same commutativity class. For example, consider the words $\w=s_1s_3s_2$ and $\w'=s_3s_1s_2$ in $W(A_3)$. Even though $\w\sim\w'$, the heaps $\HHH(\w)$ and $\HHH(\w')$ are different. We would like to say that they are ``the same,'' and we can do this using the concept of a heap isomorphism from \cite{green2010combinatorics}. Let $\HEAP$ be the category of heaps, where morphisms are defined below. 

\begin{defn} \label{defn:heap-morphism}
  A \emph{morphism} from one heap $(P,\Gamma,\phi)$ to another $(P',\Gamma',\phi')$ is a pair $(\sigma,\gamma)$, where $\sigma\colon P\to P'$ is a poset morphism and $\gamma\colon\Gamma\to\Gamma'$ is a graph homomorphism, satisfying $\gamma\circ\phi=\phi'\circ\sigma$:
  \[
  \xymatrix{ P \ar[rr]^\phi \ar[d]_\sigma && \Gamma \ar[d]^\gamma \\ P'\ar[rr]_{\phi'} && \Gamma'}
  \]
\end{defn}

If $(\sigma,\gamma)$ is a heap morphism with $\gamma$ being the identity map on $\Gamma=\Gamma'$, and $\sigma$ is injective, then $(P',\Gamma',\phi')$ is a \emph{subheap} of $(P,\Gamma,\phi)$. If $\sigma$ and $\gamma$ are both bijective, then the two heaps are \emph{isomorphic}. If we want to only consider heaps over a fixed graph $\Gamma$, which is often the case, we can define $\HEAP(\Gamma)$. We will mostly refrain from the category theory point of view in this paper, because the focus is more on Coxeter theory. A thorough categorical treatment of heaps and toric heaps will be done in a forthcoming paper. However, in order to speak about morphisms in Coxeter theory where $\Gamma\neq\Gamma'$, we would need to be clear on how to define a homomorphism between Coxeter graphs, especially regarding edge weights. We will not do that here because we will not be using it.

\begin{prop}\label{prop:isomorphic-heaps}
 If $\w\sim\w'$ are reduced words in $(W,S)$, then the heaps $\HHH(\w)$ and $\HHH(\w')$ are isomorphic. 
\end{prop}

\begin{proof}
Since $\w$ and $\w'$ differ by a sequence of short braid relations, it suffices to consider the case when they differ by a single adjacent transposition $s_{x_i}s_{x_{i+1}}\leftrightarrow s_{x_{i+1}}s_{x_i}$. In this case, the heap isomorphism is $(\sigma,\gamma)$, where the transposition $\sigma=(i\;\,i\!+\!1)$ is a poset isomorphism and $\gamma$ is the identity. 
\end{proof}

Henceforth, we will always speak of heaps up to isomorphism. By Proposition~\ref{prop:isomorphic-heaps}, each commutativity class of $w$ has a unique heap. Therefore, if $w$ is FC, then we may speak of $\HHH(w):=\HHH(\w)$ as the heap of the group element $w$. In contrast, for non-FC elements, different commutativity classes generally give non-isomorphic heaps. Our running example illustrates this nicely.

\begin{runningcont}%\label{ex:b2-h3}
  Let us recall $w=s_3s_1s_2s_1s_2$ first as an element in the Coxeter group $W(B_2)$, and then in $W(H_3)$. In $W(B_2)$, it has two commutativity classes, and the associated heaps were shown in Figure~\ref{fig:b2}. In $W(H_3)$, the bond strength between $s_1$ and $s_2$ is increased to $m(s_1,s_2)=5$. This means that $w=s_3s_1s_2s_1s_2$ is FC, and so the only heap of this group element is the first one shown in Figure~\ref{fig:b2}.
\end{runningcont}

Heaps of reduced words in Coxeter groups were studied by Stembridge in~\cite{stembridge1996fully}, though his definition was slightly different, in that he considered $i$ and $j$ \emph{incomparable} if $s_{x_i}=s_{x_j}$. For reduced words, this makes no difference. In our setting, such an $i$ and $j$ must be comparable because the $\phi$-preimage of each edge is a chain, and also contains the preimages of two vertices, and chains in posets are closed under subsets. This is also why the vertex chain requirement is built into the definition. The advantage of our framework and this requirement will become more apparent when we formalize cyclic reducibility using toric heaps. This is also in line with Viennot's original definition \cite{viennot1986combinatoire}.

The concept of a labeled linear extension of a heap was studied in \cite{stembridge1996fully}. Here, we give an abstract definition of a more general concept in our framework. We also drop the word ``labeled'' because it is implied in the context of heaps. Say that a map $\gamma\colon\Gamma\to\Gamma'$ is an \emph{edge-inclusion} if it is the identity map between graphs on the same vertex set, and every edge in $\Gamma$ is also in $\Gamma'$. 

\begin{defn}
If $(P',\Gamma',\phi')$ is the image of a morphism from a heap $(P,\Gamma,\phi)$, where $\sigma\colon P\to P'$ is an extension, and $\gamma\colon\Gamma\to\Gamma'$ is an edge-inclusion, then we say that it is an \emph{extension} of $(P,\Gamma,\phi)$. Moreover, it is a \emph{linear extension} of heaps if $\sigma$ is a linear extension of posets. 
\end{defn}

Given a heap $(P,\Gamma,\phi)$ and a graph homomorphism $\gamma\colon\Gamma\to\Gamma'$, there need not be a poset $P'$ and a map $\sigma\colon P\to P'$ such that $(\sigma,\gamma)$ is a morphism to a heap $(P',\Gamma',\phi')$. However, there will always be (at least) one if $\gamma$ is an edge-inclusion, and it is easy to see how to construct $P'$ -- it is a poset generated by the relations in $P$ with edge chains for each additional edge in $\Gamma'$. In general, such a $P'$ is not unique because there could be a choice of how to order the elements within each new edge chain, as shown in the following simple example.

\begin{ex}\label{ex:heap-extensions}
  Let $\gamma\colon\Gamma\to\Gamma'$ be the edge-inclusion between the edgeless graph on $V=\{v_1,v_2\}$ to the complete graph. The antichain $P=\{1,2\}$ and labeling map $\phi(i)=v_i$ define a heap $(P,\Gamma,\phi)$. There are two linear extensions of $P$: let $P'$ denote the one with $1<_{P'}2$ and $P''$ the one with $2<_{P''}1$. There are canonical heap morphisms $(\sigma',\gamma)$ and $(\sigma'',\gamma)$ from $(P,\Gamma,\phi)$ to $(P',\Gamma',\phi')$ and $(P'',\Gamma',\phi'')$, respectively. Here, $\sigma'$ and $\sigma''$ both send $i\mapsto i$, and $\phi'$ and $\phi''$ both send $i\mapsto v_i$. These are shown in Figure~\ref{fig:heap-extensions}.

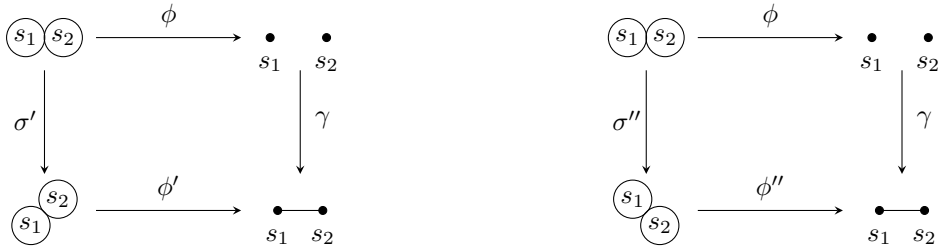
\begin{figure}[!ht]
  \tikzstyle{to} = [draw,-stealth]
  \begin{tikzpicture}[shorten >= 1pt, shorten <= 1pt,>=stealth']
    \begin{scope}[scale=1,shift={(8,0)}]
      \draw {(0,2.5) circle (0.1in) node[]{$s_1$}};
      \draw {(.5,2.5) circle (0.1in) node[]{$s_2$}};
      \draw {(.07,.35) circle (0.1in) node[]{$s_1$}};
      \draw {(.43,0) circle (0.1in) node[]{$s_2$}};
      \draw[to] (.25,2.1) to (.25,.65);
      \draw[to] (.9,2.5) to (2.9,2.5);
      \draw[to] (3.65,2.1) to (3.65,.65);
      \draw[to] (.9,.2) to (2.9,.2);
      \draw[fill=black] {(3.25,2.5) circle (1.5pt) node[label=below:$s_1$]{}};
      \draw[fill=black] {(4,2.5) circle (1.5pt) node[label=below:$s_2$]{}};
      \draw[fill=black] {(3.35,.2) circle (1.5pt) node[label=below:$s_1$]{}};
      \draw[fill=black] {(3.95,.2) circle (1.5pt) node[label=below:$s_2$]{}};
      \draw (3.35,.2) -- (3.95,.2);
      \node at (1.9,2.8) {$\phi$}; \node at (1.9,.5) {$\phi''$};
      \node at (0,1.4) {$\sigma''$}; \node at (3.95,1.4) {$\gamma$};
    \end{scope}
    \begin{scope}[scale=1,shift={(0,0)}]
      \draw {(0,2.5) circle (0.1in) node[]{$s_1$}};
      \draw {(.5,2.5) circle (0.1in) node[]{$s_2$}};
      \draw {(.07,0) circle (0.1in) node[]{$s_1$}};
      \draw {(.43,.35) circle (0.1in) node[]{$s_2$}};
      \draw[to] (.25,2.1) to (.25,.65);
      \draw[to] (.9,2.5) to (2.9,2.5);
      \draw[to] (3.65,2.1) to (3.65,.65);
      \draw[to] (.9,.2) to (2.9,.2);
      \draw[fill=black] {(3.25,2.5) circle (1.5pt) node[label=below:$s_1$]{}};
      \draw[fill=black] {(4,2.5) circle (1.5pt) node[label=below:$s_2$]{}};
      \draw[fill=black] {(3.35,.2) circle (1.5pt) node[label=below:$s_1$]{}};
      \draw[fill=black] {(3.95,.2) circle (1.5pt) node[label=below:$s_2$]{}};
      \draw (3.35,.2) -- (3.95,.2);
      \node at (1.9,2.8) {$\phi$}; \node at (1.9,.5) {$\phi'$};
      \node at (0,1.4) {$\sigma'$}; \node at (3.95,1.4) {$\gamma$};
    \end{scope}
  \end{tikzpicture}
  \caption{Two linear extensions of the same heap from
    Example~\ref{ex:heap-extensions}. 
  }\label{fig:heap-extensions}
\end{figure}
\end{ex}

Let $\LLL(\HHH)$ denote the set of linear extensions of a heap $\HHH$. In \cite{stembridge1996fully}, these are called \emph{labeled linear extensions} because they can be canonically indexed by words. For example, a linear extension $(P',\Gamma',\phi')$ can be described by the word $\phi'(x_1)\cdots\phi'(x_m)$, where $x_1<_{P'}\cdots<_{P'} x_m$.

\begin{ex}
  Let $(P,\Gamma,\phi)$ be the heap from Figure~\ref{fig:heap-example1} and $(P',\Gamma',\phi')$ the heap from Figure~\ref{fig:heap-example2}. Note that it is easy to define the labeling maps $\phi$ and $\phi'$ to adapt those heaps to our framework. Since $\sigma\colon P\to P'$ is an extension and $\gamma$ an edge-inclusion, the morphism $(\sigma,\gamma)$ is an extension of heaps. The total order $\w=acbabgdfe$ uniquely describes a heap that is a linear extension of both $(P,\Gamma,\phi)$ and $(P',\Gamma',\phi')$. 

  In the examples from Figure~\ref{fig:heap-extensions}, the two linear extensions are clearly characterized by the words $\w'=s_1s_2$ (left) and $\w''=s_2s_1$ (right).
\end{ex}

%%===========================================================================
\section{Coxeter, FC, and CFC elements}\label{sec:coxeter-fc-cfc}
%%===========================================================================

One of the goals of this paper is to develop a framework for studying what we call ``\emph{cyclic reducibility}'' in Coxeter groups. In this section, we will formalize concepts such as cyclic words and cyclic commutativity classes. In a subsequent section, we will develop a cyclic version of a heap called a \emph{toric heap}, which is essentially a labeled \emph{toric poset}. To motivate this, we will begin with conjugation of Coxeter elements, and then extend that to the cyclically fully commutative (CFC) elements. Throughout, let $(W,S)$ be a fixed Coxeter system with Coxeter graph $\Gamma$. Given words, e.g., $\c,\w,\w'$ in $S^*$, we will denote the corresponding group elements by $c,w,w'$ in $W$. 

%%-------------------------------------------------------------------------
\subsection{Conjugation of Coxeter elements}\label{subsec:coxeter-elements}
%%-------------------------------------------------------------------------

Let $(W,S)$ be a Coxeter system. A \emph{Coxeter element} is the product of all generators in some order. We denote the set of all Coxeter elements by $\C(W,S)$. Every Coxeter element of $(W,S)$ gives rise to a canonical acyclic orientation of the Coxeter graph $\Gamma$, defined by
\begin{align*}
  c\colon \Acyc(\Gamma)\longto\C(W,S),\qquad c\colon\omega\longmapsto s_{x_1}\cdots s_{x_n},
\end{align*}
where $s_{x_1}\cdots s_{x_n}$ is any linear extension of $P(\Gamma,\omega)$. It is easy to see that this map is a bijection, and so we write $c(\omega)$ to denote ``\emph{the Coxeter element defined by $\omega$}'', and $\omega(c)$ for ``\emph{the acyclic orientation given by $c$}''. Since Coxeter elements are FC, the heap poset of $c\in\C(W,S)$ does not depend on the choice of reduced word, so we may write $P_c:=P_\c$. 

\begin{prop}
  The heap poset of a Coxeter element $c$ is $P_\c=P(\Gamma,\omega(c))$ for any reduced word $\c$.
\end{prop}

\begin{proof}
  Let $\c=s_{x_1}\cdots s_{x_n}$. By construction, $P_\c=[n]$, and the labeling map $\phi_\c\colon P_\c\to\Gamma$ is defined by $\phi_\c(i)=s_{x_i}$. Since $\c$ has no repeated generators, each preimage $\phi^{-1}(s_{x_i})$ has size $1$, and so is trivially a chain. For each edge $\{s_{x_i},s_{x_j}\}$ in $\Gamma$, say $i<j$ without loss of generality, the preimage $\phi^{-1}(\{s_{x_i},s_{x_j}\})=\{i,j\}$ is a chain with $i<_{P_\c}j$. This matches the orientation of the edge $\{i,j\}$ by $\omega(c)$ from Definition~\ref{defn:w-graph}. 
\end{proof}

Since linear extensions of heaps can be indexed with words, we can consider the set $\mathcal{L}(\mathcal{H}(w))$ as a collection of words in $S^*$. A subset of $S^*$ is said to be \emph{order-theoretic} if it is the set of linear extensions of a heap. The following is a slight reformulation of Theorem 3.2 from~\cite{stembridge1996fully}.

\begin{prop}\label{prop:order-theoretic}
  For an element $w\in W$, the following are equivalent:
\begin{enumerate}[(i)]
\item $w$ is fully commutative.
\item $\RRR(w)$ is order-theoretic.
\item $\RRR(w)=\LLL(\HHH(\u))$ for some (equivalently, every) $\u \in \RRR(w)$.
\end{enumerate}
\end{prop}

  We will state and prove a cyclic analogue of this result in Theorem~\ref{thm:torically-order-theoretic}, which involves a new class of elements that are called \emph{torically fully commutative} (TFC). Before we can get there, we will first motivate the idea of cyclic reducibility with a simple example involving Coxeter elements. Observe that a cyclic shift of a Coxeter element is also a Coxeter element, and because $s=s^{-1}$ for every $s\in S$, it is also a conjugation by the initial letter:
\[
s_{x_1}(s_{x_1}s_{x_2}\cdots s_{x_m})s_{x_1}=s_{x_2}\cdots s_{x_m}s_{x_1}.
\]
On the level of acyclic orientations, these two Coxeter elements are related by converting the source vertex $s_{x_1}$ (an initial generator) into a sink (a terminal generator). This generates an equivalence relation $\equiv$ on $\Acyc(\Gamma)$, and hence on $\C(W,S)$, that we call \emph{toric equivalence}. This was first studied by Pretzel in \cite{pretzel1986reorienting} via an operation he called ``pushing down maximal vertices''. In~\cite{eriksson2009conjugacy}, H.~Eriksson and K.~Eriksson showed that these equivalence classes are in 1--1 correspondence with the conjugacy classes of $\C(W,S)$. In a recent FPSAC paper of Adin et al.\@ that introduces toric $P$-partitions, these equivalence classes are called \emph{toric DAGs} \cite{adin2018cyclic}.

\begin{thm}[\cite{eriksson2009conjugacy}]\label{thm:erikssons}
  In any Coxeter group, $c,c'\in\C(W,S)$ are conjugate if and only if $\omega(c)\equiv\omega(c')$.
\end{thm}

Thus, there are bijections between the Coxeter elements and $\Acyc(\Gamma)$, and between their conjugacy classes and toric equivalence classes, defined as follows:
\begin{align*}%\addtolength{\jot}{1em}
  \def\arraystretch{1.5}
  \begin{array}{cllllc}
    \C(W,S)\longto\Acyc(\Gamma) &&&&&
    \Conj(\C(W,S))\longto\Acyc(\Gamma)/\!\!\equiv \\
    \,\,\,\,\,\,\,\,c\longmapsto\omega(c) &&&&& \,\,\,\,\,\,\,\,\,\cl_W(c)\longmapsto[\omega(c)].
  \end{array}
\end{align*}
The sets $\Acyc(\Gamma)$ and $\Acyc(\Gamma)/\!\!\equiv$ are enumerated by the Tutte polynomial $T_\Gamma(x,y)$ at $(x,y)=(2,0)$ and $(x,y)=(1,0)$, respectively \cite{develin2016toric}.

%\begin{running}\label{run:C_4}
\begin{ex}\label{ex:C_4}
Consider the affine Coxeter group $W=W(\widetilde{A}_3)$, whose Coxeter graph is the circular graph $C_4$. Consider the following three reduced words for Coxeter elements: $\c_1=s_1s_2s_3s_4$, $\c_2=s_1s_3s_2s_4$, and $\c_3=s_1s_4s_3s_2$ (using $s_4$ instead of the usual $s_0$). The corresponding acyclic orientations are shown below.\footnote{For convenience, in the examples in this section, all of the graphs are drawn so that their vertices are the actual generators, rather than labeled bullets.} 
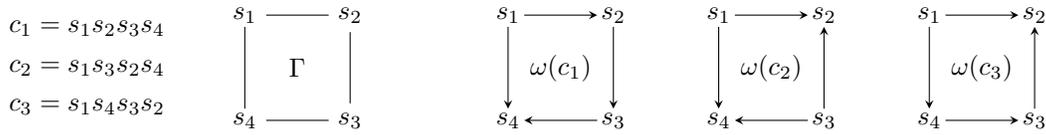
\begin{equation}\label{eq:3coxeter-elts}
  \tikzstyle{to} = [draw,-stealth]
  \tikzstyle{e} = [draw]
  \begin{tikzpicture}[scale=.7,shorten >= -2pt, shorten <= -2pt,shift={(0,0)},
    baseline=(current bounding box.center)]
    \node at (-3,1.75) {$c_1=s_1s_2s_3s_4$};
    \node at (-3,1) {$c_2=s_1s_3s_2s_4$};
    \node at (-3,.25) {$c_3=s_1s_4s_3s_2$};
    \node (s4) at (0,0) {$s_4$};
    \node (s1) at (0,2) {$s_1$};
    \node (s3) at (2,0) {$s_3$};
    \node (s2) at (2,2) {$s_2$};
    \draw (s1) -- (s2) -- (s3) -- (s4) -- (s1);
    \node at (1,1) {$\Gamma$};
    \begin{scope}[shift={(5,0)}]
      \node (s4) at (0,0) {$s_4$};
      \node (s1) at (0,2) {$s_1$};
      \node (s3) at (2,0) {$s_3$};
      \node (s2) at (2,2) {$s_2$};
      \draw[to] (s1) -- (s2);
      \draw[to] (s2) -- (s3);
      \draw[to] (s3) -- (s4);
      \draw[to] (s1) -- (s4);
      \node at (1,1) {$\omega(c_1)$};
    \end{scope}
    \begin{scope}[shift={(9,0)}]
      \node (s4) at (0,0) {$s_4$};
      \node (s1) at (0,2) {$s_1$};
      \node (s3) at (2,0) {$s_3$};
      \node (s2) at (2,2) {$s_2$};
      \draw[to] (s1) -- (s2);
      \draw[to] (s3) -- (s2);
      \draw[to] (s3) -- (s4);
      \draw[to] (s1) -- (s4);
      \node at (1,1) {$\omega(c_2)$};
    \end{scope}
    \begin{scope}[shift={(13,0)}]
      \node (s4) at (0,0) {$s_4$};
      \node (s1) at (0,2) {$s_1$};
      \node (s3) at (2,0) {$s_3$};
      \node (s2) at (2,2) {$s_2$};
      \draw[to] (s1) -- (s2);
      \draw[to] (s3) -- (s2);
      \draw[to] (s4) -- (s3);
      \draw[to] (s1) -- (s4);
      \node at (1,1) {$\omega(c_3)$};
    \end{scope}
  \end{tikzpicture}
\end{equation}
The $4!=24$ reduced words for Coxeter elements in $W(\widetilde{A}_3)$ comprise $|\!\Acyc(C_4)|=2^4-2=14$ distinct elements. The 14 acyclic orientations fall into $|\!\Acyc(C_4)/\!\!\equiv\!|=3$ toric equivalence classes: $[\omega(c_1)]$ and $[\omega(c_3)]$ have size $4$, and $[\omega(c_2)]$ has size $6$, the elements of which are shown below.
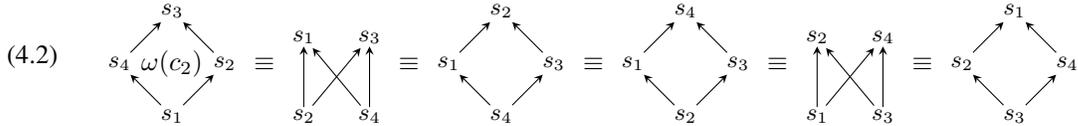
\begin{equation}\label{eq:6orientations}
    \tikzstyle{to} = [draw,-stealth]
    \tikzstyle{To} = [draw,-stealth,dashed]
    \begin{tikzpicture}[scale=.35,shorten >= -2.5pt, shorten <= -2.5pt,
        baseline=(current bounding box.center)]
      \node (3) at (0,4) {\small $s_3$};
      \node (4) at (-2,2) {\small $s_4$};
      \node (2) at (2,2) {\small $s_2$};
      \node (1) at (0,0) {\small $s_1$};
      \draw[to] (1) to (2);
      \draw[to] (1) to (4);
      \draw[to] (2) to (3);
      \draw[to] (4) to (3);
      \node at (0,2) {$\omega(c_2)$};
      \node at (3.5,2) {$\equiv$};
      \begin{scope}[shift={(5,0)}]
        \node (2) at (0,0) {\small $s_2$};
        \node (1) at (0,3) {\small $s_1$};
        \node (3) at (2.5,3) {\small $s_3$};
        \node (4) at (2.5,0) {\small $s_4$};
        \draw[to] (2) to (1);
        \draw[to] (4) to (1);
        \draw[to] (2) to (3);
        \draw[to] (4) to (3);
        \node at (4,2) {$\equiv$};
      \end{scope}
      \begin{scope}[shift={(12.5,0)}]
        \node (2) at (0,4) {\small $s_2$};
        \node (1) at (-2,2) {\small $s_1$};
        \node (3) at (2,2) {\small $s_3$};
        \node (4) at (0,0) {\small $s_4$};
        \draw[to] (4) to (1);
        \draw[to] (4) to (3);
        \draw[to] (1) to (2);
        \draw[to] (3) to (2);
        \node at (3.5,2) {$\equiv$};
      \end{scope}
      \begin{scope}[shift={(19.5,0)}]
        \node (4) at (0,4) {\small $s_4$};
        \node (1) at (-2,2) {\small $s_1$};
        \node (3) at (2,2) {\small $s_3$};
        \node (2) at (0,0) {\small $s_2$};
        \draw[to] (1) to (4);
        \draw[to] (3) to (4);
        \draw[to] (2) to (1);
        \draw[to] (2) to (3);
        \node at (3.5,2) {$\equiv$};
      \end{scope}
      \begin{scope}[shift={(24.5,0)}]
        \node (1) at (0,0) {\small $s_1$};
        \node (2) at (0,3) {\small $s_2$};
        \node (4) at (2.5,3) {\small $s_4$};
        \node (3) at (2.5,0) {\small $s_3$};
        \draw[to] (1) to (2);
        \draw[to] (1) to (4);
        \draw[to] (3) to (2);
        \draw[to] (3) to (4);
        \node at (4,2) {$\equiv$};
      \end{scope}
      \begin{scope}[shift={(32,0)}]
        \node (1) at (0,4) {\small $s_1$};
        \node (2) at (-2,2) {\small $s_2$};
        \node (4) at (2,2) {\small $s_4$};
        \node (3) at (0,0) {\small $s_3$};
        \draw[to] (3) to (2);
        \draw[to] (3) to (4);
        \draw[to] (2) to (1);
        \draw[to] (4) to (1);
      \end{scope}
  \end{tikzpicture}
\end{equation}

By Theorem~\ref{thm:erikssons}, the 14 Coxeter elements fall into 3 distinct conjugacy classes, and any two conjugate Coxeter elements differ only by cyclic shifts and short braid relations. This can be visualized by writing the reduced words as \emph{circular words}, and allowing the usual braid relations:
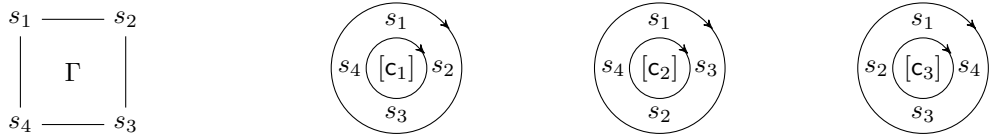
\[
  \begin{tikzpicture}
    \begin{scope}[shift={(-5,-.75)},>=stealth',scale=.7]
      \node (s4) at (0,0) {$s_4$};
      \node (s1) at (0,2) {$s_1$};
      \node (s3) at (2,0) {$s_3$};
      \node (s2) at (2,2) {$s_2$};
      \draw (s1) -- (s2) -- (s3) -- (s4) -- (s1);
      \node at (1,1) {$\Gamma$};
    \end{scope}
    \begin{scope}[shift={(0,0)},>=stealth',scale=.62]
      \draw[decoration={markings, mark=at position 0.125 with
          {\arrow{<}}}, postaction={decorate}] (0,0) circle(1.4cm);
      \draw[decoration={markings, mark=at position 0.125 with
          {\arrow{<}}}, postaction={decorate}] 
      (0cm,0cm) circle(0.65cm); 
      \draw (90:1) node{$s_1$}; 
      \draw (0:1) node{$s_2$}; 
      \draw (-90:1) node{$s_3$}; 
      \draw (-180:1) node{$s_4$};
      \node at (0,0) {$[\c_1]$};
    \end{scope}
    \begin{scope}[shift={(3.5,0)},>=stealth',scale=.62]
      \draw[decoration={markings, mark=at position 0.125 with
          {\arrow{<}}}, postaction={decorate}] (0,0) circle(1.4cm);
      \draw[decoration={markings, mark=at position 0.125 with
          {\arrow{<}}}, postaction={decorate}] 
      (0cm,0cm) circle(0.65cm); 
      \draw (90:1) node{$s_1$}; 
      \draw (0:1) node{$s_3$}; 
      \draw (-90:1) node{$s_2$}; 
      \draw (-180:1) node{$s_4$};
      \node at (0,0) {$[\c_2]$};
    \end{scope}
    \begin{scope}[shift={(7,0)},>=stealth',scale=.62]
      \draw[decoration={markings, mark=at position 0.125 with
          {\arrow{<}}}, postaction={decorate}] (0,0) circle(1.4cm);
      \draw[decoration={markings, mark=at position 0.125 with
          {\arrow{<}}}, postaction={decorate}] 
      (0cm,0cm) circle(0.65cm); 
      \draw (90:1) node{$s_1$}; 
      \draw (0:1) node{$s_4$}; 
      \draw (-90:1) node{$s_3$}; 
      \draw (-180:1) node{$s_2$};
      \node at (0,0) {$[\c_3]$};
    \end{scope}
  \end{tikzpicture}
  \]
In the circular words $[\c_1]$ and $[\c_3]$ above, no short braid relations can be applied, so both $c_1$ and $c_3$ are conjugate to only $4$ Coxeter elements each -- all cyclic shifts. In contrast, the relations $s_1s_3=s_3s_1$ and $s_2s_4=s_4s_2$ can be applied to $[\c_2]$, yielding three other ``reduced cyclic words'':

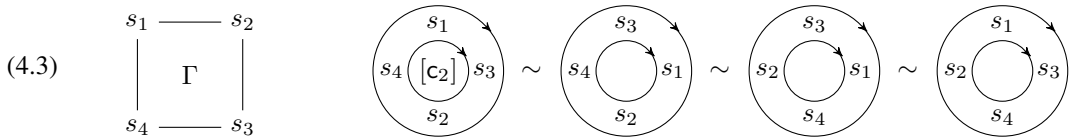
\begin{equation}\label{eq:4words}
  \begin{tikzpicture}[baseline=(current bounding box.center)]
    \begin{scope}[shift={(-4,-.75)},>=stealth',scale=.7]
      \node (s4) at (0,0) {$s_4$};
      \node (s1) at (0,2) {$s_1$};
      \node (s3) at (2,0) {$s_3$};
      \node (s2) at (2,2) {$s_2$};
      \draw (s1) -- (s2) -- (s3) -- (s4) -- (s1);
      \node at (1,1) {$\Gamma$};
    \end{scope}
    \begin{scope}[shift={(0,0)},>=stealth',scale=.62]
      \draw[decoration={markings, mark=at position 0.125 with
          {\arrow{<}}}, postaction={decorate}] (0,0) circle(1.4cm);
      \draw[decoration={markings, mark=at position 0.125 with
          {\arrow{<}}}, postaction={decorate}] 
      (0cm,0cm) circle(0.65cm); 
      \draw (90:1) node{$s_1$}; 
      \draw (0:1) node{$s_3$}; 
      \draw (-90:1) node{$s_2$}; 
      \draw (-180:1) node{$s_4$};
      \node at (0,0) {$[\c_2]$};
      \node at (2,0) {$\sim$};
    \end{scope}
    \begin{scope}[shift={(2.5,0)},>=stealth',scale=.62]
      \draw[decoration={markings, mark=at position 0.125 with
          {\arrow{<}}}, postaction={decorate}] (0,0) circle(1.4cm);
      \draw[decoration={markings, mark=at position 0.125 with
          {\arrow{<}}}, postaction={decorate}] 
      (0cm,0cm) circle(0.65cm); 
      \draw (90:1) node{$s_3$}; 
      \draw (0:1) node{$s_1$}; 
      \draw (-90:1) node{$s_2$}; 
      \draw (-180:1) node{$s_4$};
      \node at (2,0) {$\sim$};
    \end{scope}
    \begin{scope}[shift={(5,0)},>=stealth',scale=.62]
      \draw[decoration={markings, mark=at position 0.125 with
          {\arrow{<}}}, postaction={decorate}] (0,0) circle(1.4cm);
      \draw[decoration={markings, mark=at position 0.125 with
          {\arrow{<}}}, postaction={decorate}] 
      (0cm,0cm) circle(0.65cm); 
      \draw (90:1) node{$s_3$}; 
      \draw (0:1) node{$s_1$}; 
      \draw (-90:1) node{$s_4$}; 
      \draw (-180:1) node{$s_2$};
      \node at (2,0) {$\sim$};
    \end{scope}
    \begin{scope}[shift={(7.5,0)},>=stealth',scale=.62]
      \draw[decoration={markings, mark=at position 0.125 with
          {\arrow{<}}}, postaction={decorate}] (0,0) circle(1.4cm);
      \draw[decoration={markings, mark=at position 0.125 with
          {\arrow{<}}}, postaction={decorate}] 
      (0cm,0cm) circle(0.65cm); 
      \draw (90:1) node{$s_1$}; 
      \draw (0:1) node{$s_3$}; 
      \draw (-90:1) node{$s_4$}; 
      \draw (-180:1) node{$s_2$};
    \end{scope}
  \end{tikzpicture}
  \end{equation}

There are six distinct Coxeter elements, and 16 reduced words, that can arise from these four cyclic words, assuming they are read off clockwise:
\begin{equation}\label{eq:16elements}
\arraycolsep=1.4pt
\begin{array}{rrrrrr}
  s_1s_2s_4s_3\,\quad\;\;& s_2s_4s_1s_3\,\quad\;\;& s_4s_1s_3s_2\,\quad\;\;&
  s_2s_1s_3s_4\,\quad\;\;& s_1s_3s_2s_4\,\quad\;\;& s_3s_2s_4s_1\, \\
  =s_1s_4s_2s_3\,\quad\;\;& =s_2s_4s_3s_1\,\quad\;\;& =s_4s_3s_1s_2\,\quad\;\;&
  =s_2s_3s_1s_4\,\quad\;\;& =s_1s_3s_4s_2\,\quad\;\;& =s_3s_4s_2s_1\, \\
                 \quad\;\;& =s_4s_2s_3s_1\,\quad\;\;& &
  & =s_3s_1s_4s_2\,\quad\;\;& \\
                 \quad\;\;& =s_4s_2s_1s_3\,\quad\;\;& &
  & =s_3s_1s_2s_4\,\quad\;\;& \\
\end{array}
\end{equation}
The elements in the $i^{\rm th}$ column above are the linear extensions of the poset defined by the $i^{\rm th}$ orientation in Eq.~\eqref{eq:6orientations}. 
\end{ex}
%\end{running}

Example~\ref{ex:C_4} should motivate the value of developing a theory of cyclic reducibility in Coxeter groups. For example, the four cyclic words in Eq.~\eqref{eq:4words} should be thought of as lying in the ``cyclic commutativity class'' containing the ``cyclic word'' $[\c_2]$. In Section~\ref{sec:cyclic-reducibility}, we will develop this framework. But first, we need to focus on the curious ``cyclic partial order'' structure that arises. Cyclic words under the equivalence generated by short braid relations are like cyclic analogues of traces, though without the monoid structure, because there is no canonical way to concatenate cyclic words. This cyclic poset structure can be formalized via \emph{toric posets}~\cite{develin2016toric}, which leads to the concept of a \emph{toric heap}. This is essentially a labeled toric poset, in the same sense of how ordinary heaps are labeled ordinary posets. It allows us to extend the examples shown in this section far beyond just Coxeter elements, which we will do next.

%%-------------------------------------------------
\subsection{Cyclically fully commutative (CFC) elements}
%%-------------------------------------------------

Now that we have seen the interplay between Coxeter elements, acyclic orientations, and heaps, and how they behave under conjugacy, we will extend these ideas to a larger class of elements. This will elucidate the key structural properties as well as motivate the main ideas of our cyclic reducibility framework.

It is well known that in any Coxeter group, if $s\in S$, then $\ell(sw)=\ell(w)\pm 1$, and so $\ell(w^k)\leq k\ell(w)$. If equality holds for all $k\in\N$, then we say that $w$ is \emph{logarithmic}.\footnote{Some authors call this property \emph{straight}, which is motivated by a geometric view of Coxeter groups. We will use the term \emph{logarithmic}, as our viewpoint is more purely combinatorial.} In 2009, it was shown independently by D.~Speyer~\cite{speyer2009powers} and H.~Eriksson and K.~Eriksson~\cite{eriksson2009conjugacy} that in infinite irreducible Coxeter systems, Coxeter elements are logarithmic. It is simple to extend this to the non-irreducible case -- each connected component of $\Gamma$ must be the Coxeter graph of an infinite group. The logarithmic property was key to the Erikssons' proof of the conjugacy problem (Theorem~\ref{thm:erikssons}). Also crucial was the source-to-sink property, i.e., toric equivalence. In plain English, we mean that (i) Coxeter elements are FC (they avoid long braids), and (ii) cyclic shifts of Coxeter elements remain FC. These properties can naturally be extended beyond Coxeter elements. 

\begin{defn}
An element $w \in W$ is \emph{cyclically fully commutative} (CFC) if for any reduced word of $w$, every cyclic shift is reduced and FC.
\end{defn}

\begin{ex}\label{ex:cfc-affine}
  Figure~\ref{fig:cfc-affine} shows examples of CFC elements in two affine Coxeter groups. On the left is the Coxeter graph of the group $W(\widetilde{C}_4)$ and the CFC element $w_1=s_0s_1s_2s_3s_4s_3s_2s_1$ drawn in a circle so the reader can visually see how there are no long braids. To the right is the Coxeter graph of the group $W(\widetilde{E}_6)$ and the CFC element $w_2=s_1s_3s_2s_4s_3s_5s_4s_6s_0s_3s_2s_6$, also drawn in a circle. 
\end{ex}

    \begin{figure}[!ht]
      \centering
      \begin{tikzpicture}
    \begin{scope}[shift={(-.75,3.25)},>=stealth',scale=.65]
      \draw[decoration={markings, mark=at position 0.125 with
          {\arrow{<}}}, postaction={decorate}] (0,0) circle(1.4cm);
      \draw[decoration={markings, mark=at position 0.125 with
          {\arrow{<}}}, postaction={decorate}] 
      (0cm,0cm) circle(0.65cm); 
      \draw (90:1) node{$s_0$}; \draw (45:1) node{$s_1$}; 
      \draw (0:1) node{$s_2$}; \draw (-45:1) node{$s_3$}; 
      \draw (-90:1) node{$s_4$}; \draw (-135:1) node{$s_3$}; 
      \draw (-180:1) node{$s_2$}; \draw (135:1) node{$s_1$};
      \node at (0,0) {$[\w_1\!]$};
    \end{scope}
      \begin{scope}[scale=.9,shift={(0,3)}]
        \draw[fill=black] \foreach \x in {1,2,...,5} {(\x,0) circle (1.5pt)};
        \draw {(1,0) node[label=below:$s_0$]{}
          (2,0) node[label=below:$s_1$]{}
          (3,0) node[label=below:$s_2$]{}
          (4,0) node[label=below:$s_3$]{}
          (5,0) node[label=below:$s_4$]{}
          [-] (1,0) -- (5,0)
        };
        \draw (1.5,0) node[label=above:$4$]{};
        \draw (4.5,0) node[label=above:$4$]{};
        \node at (3,1.25) {$\Gamma(\widetilde{C}_4)$};
      \end{scope}
        \begin{scope}[scale=.9,shift={(6,3)}]
        \draw[fill=black] \foreach \y in {0.75,1.5} {(3,\y) circle (1.5pt)};
        \draw[fill=black] \foreach \x in {1,2,...,5} {(\x,0) circle (1.5pt)};
        \draw {(1,0) node[label=below:$s_1$]{}
          (2,0) node[label=below:$s_2$]{}
          (3,0) node[label=below:$s_3$]{}
          (4,0) node[label=below:$s_4$]{}
          (5,0) node[label=below:$s_5$]{}
          (3,0.75) node[label=right:$s_6$]{}
          (3,1.5) node[label=right:$s_0$]{}
          [-] (1,0) -- (5,0)
          [-] (3,0) -- (3,1.5)};
        \node at (1.5,1.25) {$\Gamma(\widetilde{E}_6)$};
        \end{scope}
    \begin{scope}[shift={(12,3.35)},>=stealth',scale=.75]
      \draw[decoration={markings, mark=at position 0.125 with
          {\arrow{<}}}, postaction={decorate}] (0,0) circle(1.4cm);
      \draw[decoration={markings, mark=at position 0.125 with
          {\arrow{<}}}, postaction={decorate}] 
      (0cm,0cm) circle(0.65cm); 
      \draw(90:1)node{$s_1$}; \draw(60:1)node{$s_3$}; \draw(30:1)node{$s_2$}; 
      \draw(0:1)node{$s_4$}; \draw(-30:1)node{$s_3$}; \draw(-60:1)node{$s_5$};
      \draw(-90:1)node{$s_4$}; \draw(-120:1)node{$s_6$}; \draw(-150:1)node{$s_0$}; 
      \draw(-180:1)node{$s_3$}; \draw(150:1)node{$s_2$}; \draw(120:1)node{$s_6$};
      \node at (0,0) {$[\w_2]$};
    \end{scope}
      \end{tikzpicture}
      \caption{CFC elements in affine Coxeter groups, drawn in a circle to highlight the absence of long braids, as well as to motivate concepts such as ``cyclic words'' and ``cyclic commutativity classes'', which toric heaps will allow us to formalize.}
      \label{fig:cfc-affine}
    \end{figure}
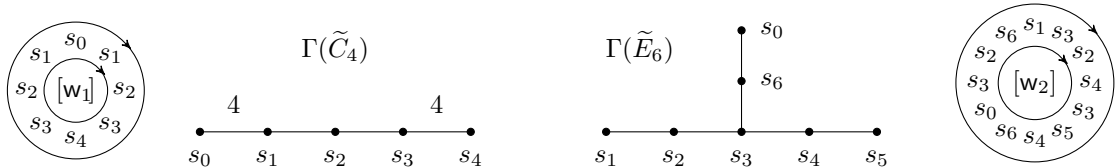

CFC elements were introduced and studied in 2012 by Boothby et al.~\cite{boothby2012cyclically}. They have recently been characterized and enumerated in all affine Coxeter groups, and their generating functions were shown to be rational in all Coxeter groups \cite{petreolle2016generating,petreolle2017characterization}.

%%===========================================================================
\section{Labeled toric posets and toric heaps}\label{sec:toric}
%%===========================================================================

%%----------------------------------------------------------------------------
\subsection{Posets and toric posets, geometrically}\label{subsec:toric-posets}
%%----------------------------------------------------------------------------

Throughout this section, $\Gamma$ is a Coxeter graph, $G=(V,E)$ is an undirected graph, $\Acyc(G)$ is the set of acyclic orientations of $G$, and $\equiv$ is \emph{toric equivalence}, i.e., the equivalence relation on $\Acyc(G)$ generated by source-to-sink conversions.

A toric poset should be thought of as a cyclic version of a poset. The most concrete way to define them are as toric equivalence classes of $\Acyc(G)$, and we write this as, e.g., $T(G,[\omega])$. However, this has two significant drawbacks: first, it suggests a dependence on the graph $G$, which is a little misleading. Recall that we can similarly define a poset $P=P(G,\omega)$ on a graph, though in actuality, $G$ is almost never uniquely determined. Specifically, there is a unique minimal graph (the Hasse diagram, $\hat{G}^{\Hasse}(P)$), and a unique maximal graph (the transitive closure, $\bar{G}(P)$) on $V$ such that any $G'=(V,E')$ whose edge set $E'$ is between the edge sets of these two extremes (with respect to subset inclusion) will work. Of course, care must also be taken with how to define $\omega'\in\Acyc(G')$ so that $P(G,\omega)=P(G',\omega')$, but that is straightforward -- any shared edges must be oriented the same way. 

The second drawback of the notation $T(G,[\omega])$ is that it obscures the ``more proper'' geometric way to define a toric poset. We will motivate this by revisiting the geometric interpretation of $P(G,\omega)$. In fact, finite posets can be \emph{defined} and developed purely geometrically, as chambers of graphic hyperplane arrangements. For distinct vertices $i$ and $j$ in $V$, let $H_{ij}$ be the hyperplane $x_i=x_j$ in $\R^V$. The graphic arrangement of $G$ is the set $\AAA(G)=\{H_{ij}\mid \{i,j\}\in E\}$. Each point $x=(x_1,\dots,x_n)$ in the complement $\R^V\setmin\AAA(G)$ determines a canonical acyclic orientation $\omega(x)$, by directing the edge $\{i,j\}$ as $i\to j$ if $x_i<x_j$. The fibers of this mapping are the chambers of the $\AAA(G)$, and so this induces a bijection between chambers of $\AAA(G)$ and acyclic orientations of $G$:
  \[
  \xymatrix{\R^V \setmin \AAA(G) \ar@{>>}[dr] \ar[rr]^{\alpha_G} & & \Acyc(G) \\
    & \Chambers \AAA(G) \ar@{.>}[ur] & \\ }
  \]
At this point, one could define a poset to be any subset of $\R^V$ that arises as a chamber of some graphic hyperplane arrangement. This removes the reference to a particular graph in the definition. Given a poset $P$, we write $c(P)$ for the chamber in $\R^V$ determined by $P$. Given a chamber $c$, we write $P(c)$ for the poset determined by $c$.

Though this geometric perspective is a bit superfluous for ordinary posets, it is absolutely necessary for toric posets, where it is not so clear how to pull apart the concept of a toric poset from the underlying graph. The geometric definition of a toric poset arises from the observation that if we first quotient out $\R^V$ by the integer lattice $\Z^V$, then converting a source $x_i$ into a sink corresponds to crossing a coordinate hyperplane $x_i=0$, and this does not change the corresponding connected component in the torus $\R^V\!/\Z^V$. An example of this is shown in Figure~\ref{fig:K_3}.

This leads us to our first of five ``Definition / Theorems''. We use that term because each of them involves a non-trivial statement or equivalence proven in \cite{develin2016toric}, which allows us to take them as the definition in this paper. 

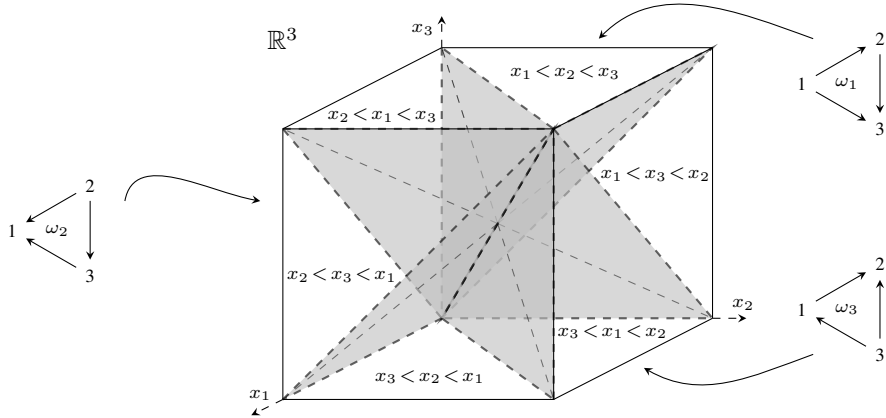
\begin{figure}[!ht]
  \tikzstyle{axis} = [draw, dashed,-stealth]
  \tikzstyle{to} = [draw,-stealth]
  \tikzstyle{v} = [circle, draw, fill=white,inner sep=0pt, 
    minimum size=1mm]  
  \begin{tikzpicture}[scale=0.6]
    \filldraw[thick,dashed,fill=lightgray,opacity=0.6] 
    (3.52,7.8)--(3.52,1.8)--(6,6)--cycle;
    \draw[dashed,opacity=0.7] {(4.76,3.9)--(3.52,7.8)};
    \filldraw[thick,dashed,fill=lightgray,opacity=0.6] 
    (6,6)--(9.52,7.8)--(3.52,1.8)--cycle;
    \draw[dashed,opacity=0.5] {(4.76,3.9)--(9.52,7.8)};
    \filldraw[thick,dashed,fill=lightgray,opacity=0.6] 
    (6,6)--(9.52,1.8)--(3.52,1.8)--cycle;
    \draw[dashed,opacity=0.5] {(4.76,3.9)--(9.52,1.8)};
    \filldraw[thick,dashed,fill=lightgray,opacity=0.6] 
    (0,6)--(6,6)--(3.52,1.8)--cycle;
    \draw[dashed,opacity=0.5] {(4.76,3.9)--(0,6)};
    \filldraw[thick,dashed,fill=lightgray,opacity=0.6] 
    (6,0)--(6,6)--(3.52,1.8)--cycle;
    \draw[dashed,opacity=0.5] {(4.76,3.9)--(6,0)};
    \filldraw[thick,dashed,fill=lightgray,opacity=0.6]%,pattern=north east lines] 
    (0,0)--(6,6)--(3.52,1.8)--cycle;
    \draw[dashed,opacity=0.5] {(4.76,3.9)--(0,0)};
    \draw {(0,0)--(0,6)--(6,6)--(6,0)--(0,0)};
    \draw {(3.52,7.8)--(9.52,7.8)--(9.52,1.8)};
    \draw {(0,6)--(3.52,7.8)}; 
    \draw {(6,0)--(9.52,1.8)}; 
    \draw {(6,6)--(9.52,7.8)};
    \draw (7.3,1.5) node{\scriptsize $x_3\!<\!x_1\!<\!x_2$};
    \draw (8.25,5) node{\scriptsize $x_1\!<\!x_3\!<\!x_2$};
    \draw (3.25,.5) node{\scriptsize $x_3\!<\!x_2\!<\!x_1$};
    \draw (6.25,7.25) node{\scriptsize $x_1\!<\!x_2\!<\!x_3$};
    \draw (1.3,2.75) node{\scriptsize $x_2\!<\!x_3\!<\!x_1$};
    \draw (2.2,6.3) node{\scriptsize $x_2\!<\!x_1\!<\!x_3$};
    \draw [axis] (9.52,1.8) to (10.27,1.8); \node at(10.2,2.1){\scriptsize $x_2$};
    \draw [axis] (3.52,7.8) to (3.52,8.5); \node at(3.1,8.2){\scriptsize $x_3$};
    \draw [axis] (0,0) to (-.7,-.35); \node at(-.5,.1){\scriptsize $x_1$};
    \node at (0,8) {$\R^3$};
    \begin{scope}[shift={(11.5,7)}]
      \node (1a) at (0,0) {\scriptsize 1};
      \node (2a) at (1.73,1) {\scriptsize 2}; 
      \node (3a) at (1.73,-1) {\scriptsize 3}; 
      \draw[to] (1a) to (2a);
      \draw[to] (2a) to (3a);
      \draw[to] (1a) to (3a);
      \path[to,out=160,in=45] (.25,1) to (-4.5,1);
      \node at (1,0) {\scriptsize $\omega_1$};
    \end{scope}
    \begin{scope}[shift={(11.5,2)}]
      \node (1c) at (0,0) {\scriptsize 1};
      \node (2c) at (1.73,1) {\scriptsize 2}; 
      \node (3c) at (1.73,-1) {\scriptsize 3}; 
      \draw[to] (3c) to (2c);
      \draw[to] (3c) to (1c);
      \draw[to] (1c) to (2c);
      \path[to,out=200,in=-30] (.25,-1) to (-3.5,-1.3);
      \node at (1,0) {\scriptsize $\omega_3$};
    \end{scope}
    \begin{scope}[shift={(-6,3.75)}]
      \node (1b) at (0,0) {\scriptsize 1};
      \node (2b) at (1.73,1) {\scriptsize 2}; 
      \node (3b) at (1.73,-1) {\scriptsize 3}; 
      \draw[to] (2b) to (3b);
      \draw[to] (3b) to (1b);
      \draw[to] (2b) to (1b);
      \path[to,out=80,in=170] (2.5,.65) to (5.5,.65);
      \node at (1,0) {\scriptsize $\omega_2$};
    \end{scope}
  \end{tikzpicture}
  \caption{The hyperplane arrangement $\AAA(G)$ of the complete graph $G=K_3$. The toric arrangement $\AAA_{\tor}(G)$ is achieved by identifying opposite sides of the unit cube. Doing this merges the three regions corresponding to the three acyclic orientations shown into one single toric chamber.}\label{fig:K_3}
\end{figure}

\begin{defn-thm}\label{defn-thm:toric-poset}
  A \emph{toric poset} over $G$ is characterized by either:
  \begin{enumerate}[(i)]
  \item an equivalence class $[\omega]$ in $\Acyc(G)/\!\!\equiv$,
  \item a chamber $c$ of the toric hyperplane arrangement $\AAA_{\tor}(G)$
    in $\R^V\!/\Z^V$.
  \end{enumerate}
  We will write $T(G,[\omega])$ to mean the toric poset characterized
  by $[\omega]$ in $\Acyc(G)/\!\!\equiv$. Given a toric poset $T$, we
  write $c(T)$ when we wish to speak of the chamber in $\R^V\!/\Z^V$
  determined by $T$, and given a chamber $c$, we write $T(c)$ to emphasize the
  toric poset determined by it.
\end{defn-thm}

Aside from the definition not depending on a distinguished graph $G$, the second advantage to the geometric perspective is the recurring theme that many standard features of ordinary posets, such as chains, antichains, Hasse diagrams, transitive closure, order ideals, and so on, have natural toric analogues. However, it is usually not clear how these should be defined in terms of an equivalence class of acyclic orientations. Instead, the natural definition often only becomes apparent when one interprets the classical definition geometrically, and then passes to the quotient $\pi\colon\R^V\longto\R^V\!/\Z^V$, as illustrated by the following commutative diagram: \vspace{-1mm}
  \[
  \xymatrix{ \R^V\setmin\AAA(G) \ar[r]^\pi\ar[d]_{\alpha_G} &
    \R^V\!/\Z^V\setmin\AAA_{\tor}(G) \ar[d]^{\bar{\alpha}_G} \\ \Acyc(G)
    \ar@{.>}[r]^{???} & \Acyc(G)/\!\!\equiv}
  \]
Often, this toric analogue then has a natural combinatorial interpretation in terms of directed graphs, which usually ends up being more convenient. A list of these can be found in the Introduction of \cite{macauley2016morphisms}.

For an example of this, consider a non-edge $\{i,j\}\not\in E$. One might ask whether $i$ and $j$ are comparable in $P=P(G,\omega)$. In other words, would adding $\{i,j\}$ to $G$ force its orientation by $\omega$, which happens when one of the two ways to orient it would create a directed cycle? If so, $\{i,j\}$ is \emph{implied by transitivity}, and either $i\leq_P j$ or $j\leq_P i$. Geometrically, this means that the hyperplane $H_{ij}$ is disjoint from the chamber $c(P)$ of $\AAA(G)$. The combinatorial condition is also straightforward: $\{i,j\}$ is implied by transitivity if and only if $i$ and $j$ lie on a common directed path in $\omega$ (i.e., lie on a chain in $P$). 

There is a toric analogue of transitivity, which is easy to state geometrically: the edge $\{i,j\}$ is \emph{implied by toric transitivity} if the toric hyperplane $H_{ij}^{\tor}$ is disjoint from the chamber $c(T(G,[\omega]))$ of $\AAA_{\tor}(G)$. The combinatorial condition of this is less clear due to the absence of a binary relation, but luckily, it has a simple answer, which is basically just adding the word ``toric'' to the ordinary case. Specifically, a \emph{toric directed path} in $\omega$ is a directed path $i_1\to\cdots\to i_k$ such that the edge $i_1\to i_k$ is also present, as shown in Figure~\ref{fig:toric-chain}. It was shown in \cite{develin2016toric} that if $i_1\to\cdots\to i_k$ is a toric directed path in $\omega$, then some cyclic shift is a toric directed path in $\omega'$ for each $\omega'\equiv\omega$. 

 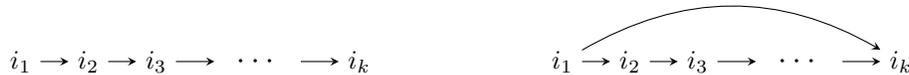
\begin{figure}[!ht]
  \tikzstyle{to} = [draw,-stealth]
  \begin{tikzpicture}[scale=.9]
    \begin{scope}
      \node (i1) at (1,0) {$i_1$};
      \node (i2) at (2,0) {$i_2$};
      \node (i3) at (3,0) {$i_3$};
      \node (i4) at (4,0) {};
      \node (i5) at (5,0) {};
      \node at (4.5,0) {\Large $\cdots$};
      \node (ik) at (6,0) {$i_k$};
      \draw[to] (i1) to (i2);
      \draw[to] (i2) to (i3);
      \draw[to] (i3) to (i4);
      \draw[to] (i5) to (ik);
    \end{scope}
    \begin{scope}[shift={(8,0)}]
      \node (i1) at (1,0) {$i_1$};
      \node (i2) at (2,0) {$i_2$};
      \node (i3) at (3,0) {$i_3$};
      \node (i4) at (4,0) {};
      \node (i5) at (5,0) {};
      \node at (4.5,0) {\Large $\cdots$};
      \node (ik) at (6,0) {$i_k$};
      \draw[to] (i1) to (i2);
      \draw[to] (i2) to (i3);
      \draw[to] (i3) to (i4);
      \draw[to] (i5) to (ik);
      \draw[to,bend left=30] (i1) to (ik); 
    \end{scope}
  \end{tikzpicture}
\caption{A chain in a poset is any subset that lies on a directed path (left). A toric chain in a toric poset is any subset that lies on a toric directed path (right).}\label{fig:toric-chain}
\end{figure}

It is worth noting that just like directed paths in ordinary posets, a two-element toric directed path is an edge, and all vertices are one-element toric directed paths. We will say that the empty set is vacuously a directed path and a toric directed path. 

\begin{defn-thm}
  A non-edge $\{i,j\}$ is \emph{implied by transitivity} in $P(G,\omega)$ if and only if $i$ and $j$ lie on a directed path in $\omega$. 

  A non-edge $\{i,j\}$ is \emph{implied by toric transitivity} in $T(G,[\omega])$ if and only if $i$ and $j$ lie on a toric directed path in $\omega$. 
\end{defn-thm}

For ordinary posets, adding all edges implied by transitivity, in any order, yields the transitive closure, $\bar{G}(P(G,\omega))$. The \emph{toric transitive closure} $\bar{G}^{\tor}(T(G,[\omega]))$ can be defined analogously.

For ordinary posets, removing all unnecessary edges, in any order, yields the Hasse diagram, denoted $\hat{G}^{\Hasse}(P(G,\omega))$. Geometrically, this just means removing all hyperplanes $H_{ij}$ that are disjoint from the chamber $c(P(G,\omega))$. The \emph{toric Hasse diagram} is completely analogous, and denoted $\hat{G}^{\torHasse}(T(G,[\omega]))$. However, the Hasse diagram of $P(G,\omega)$ and the toric Hasse diagram of $T(G,[\omega])$ are generally not the same. To see why, we need the notion of a toric chain.

A key concept behind the aforementioned features of posets is that of a \emph{chain}, which is a totally ordered set. This can be characterized geometrically in terms of the coordinates of the entries of all points in the corresponding chamber, or combinatorially as a subset of vertices lying on a directed path in $\omega$. A toric chain in $T(G,[\omega])$ is a \emph{totally cyclically ordered set}. This can also be characterized geometrically in terms of the coordinates of the entries of all points in the corresponding toric chambers. Luckily, its combinatorial characterization is both simple and analogous to the ordinary poset case. 

\begin{defn-thm}
  A set $C=\{i_1,\dots,i_k\}\subseteq V$ is a:
  \begin{itemize}
  \item \emph{chain} of $P(G,\omega)$ if it lies on a directed path in $\omega$;
  \item \emph{toric chain} of $T(G,[\omega])$ if it lies on a toric directed path in $\omega$.
  \end{itemize}
  Both chains and toric chains are closed under subsets. 
\end{defn-thm}

\begin{ex}\label{ex:4-vertex-graphs}
  Let $L_4$, $C_4$, and $K_4$ be the line, circular, and complete graph on $4$ vertices, respectively. Assume that the vertices are ordered ``naturally,'' i.e., they all contain (at least) the edges $\{1,2\}$, $\{2,3\}$, and $\{3,4\}$. Define $\omega\in\Acyc(L_4)$, $\omega'\in\Acyc(C_4)$, and $\omega''\in\Acyc(K_4)$ so that $\{i,j\}$ is oriented $i\to j$ if $i<j$. Then
  \[
  \begin{array}{l}
  \hat{G}^{\Hasse}(P(K_4,\omega''))=L_4, \vspace{1mm} \\
  \hat{G}^{\torHasse}(T(K_4,[\omega'']))=C_4,
  \end{array}
  \hspace{10mm}
  \begin{array}{l}
    \bar{G}(P(L_4,\omega))=K_4=\bar{G}^{\tor}(T(C_4,[\omega'])), \vspace{1mm} \\
    \bar{G}^{\tor}(T(L_4,[\omega]))=L_4.
  \end{array}
  \]
  Figure~\ref{fig:transitivity} shows a visual of these examples. The solid lines show the edges that make up the Hasse and toric Hasse diagram. The transitive closure and toric transitive closure are given by including the (undirected) dashed edges as well.

  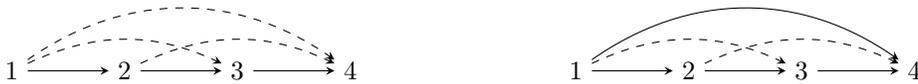
\begin{figure}[!ht]
  \begin{tikzpicture}
    \tikzstyle{to} = [draw,-stealth]
    \begin{scope}[scale=1.5,shift={(0,0)}]
      \node (v1) at (1,0) {$1$};
      \node (v2) at (2,0) {$2$};
      \node (v3) at (3,0) {$3$};
      \node (v4) at (4,0) {$4$};
      \draw[to] (v1) to (v2); \draw[to] (v2) to (v3); \draw[to] (v3) to (v4);
      \draw[to,bend left=25,dashed] (v1) to (v3); 
      \draw[to,bend left=25,dashed] (v2) to (v4);
      \draw[to,bend left=35,dashed] (v1) to (v4); 
    \end{scope}
    \begin{scope}[scale=1.5,shift={(5,0)}]
      \node (v1) at (1,0) {$1$};
      \node (v2) at (2,0) {$2$};
      \node (v3) at (3,0) {$3$};
      \node (v4) at (4,0) {$4$};
      \draw[to] (v1) to (v2); \draw[to] (v2) to (v3); \draw[to] (v3) to (v4);
      \draw[to,bend left=25,dashed] (v1) to (v3); 
      \draw[to,bend left=25,dashed] (v2) to (v4);
      \draw[to,bend left=35] (v1) to (v4); 
    \end{scope}
  \end{tikzpicture}
  \caption{\emph{Left}: The Hasse diagram of $P(L_4,\omega)$ consists of the solid edges (undirected). The dashed edges are additionally implied by transitivity. \emph{Right}: The toric Hasse diagram of $T(K_4,[\omega''])$ consists of the solid edges (undirected). The dashed edges are implied by toric transitivity. See Example~\ref{ex:4-vertex-graphs} for details.}\label{fig:transitivity}
  \end{figure}

\end{ex}

The last concepts that we need the toric analogue of are extensions and total orders. A \emph{total order} is a poset $P(K_V,\omega)$, where $K_V$ is the complete graph. Geometrically, this corresponds to a chamber of $\AAA(K_V)$. Intuitively, a total toric order is a toric poset that is totally cyclically ordered. 

\begin{defn-thm}
  A \emph{total toric order} is characterized by either:
  \begin{enumerate}[(i)]
  \item a toric poset $T(K_V,[\omega])$,
  \item a chamber of $\AAA_{\tor}(K_V)$. 
  \end{enumerate}
\end{defn-thm} 

An \emph{extension} $P'$ of a poset $P$ is characterized combinatorially by adding relations (or edges to $(G,\omega)$), or geometrically by $c(P')\subseteq c(P)$ in $\R^V$ (the result of adding hyperplanes to $\AAA(G)$). Moreover, $P'$ is a \emph{linear extension} of $P$ if it is an extension and a total order. 

\begin{defn-thm}
  Let $T=T(G,[\omega])$ be a toric poset, and assume without loss of generality that $G=(V,E)$ is its toric Hasse diagram. A \emph{toric extension} $T'$ of a toric poset $T$ is characterized by either:
  \begin{enumerate}[(i)]
  \item $c(T')\subseteq c(T)$ in $\R^V\!/\Z^V$,
  \item $T'=T(G',[\omega'])$, where $G'=(V,E')$, $E\subseteq E'$, and all edges in $E\cap E'$ are oriented the same way by $\omega$ and $\omega'$. 
  \end{enumerate}
  If $T'$ is a toric extension of $T$ and a total toric order, then it is a \emph{total toric extension}.
\end{defn-thm}

There are $|\Acyc(K_V)|=T_{K_V}(2,0)=n!$ total orders on a size-$n$ set $V$, where $T_{K_V}$ is the Tutte polynomial. In contrast, there are $|\Acyc(K_V)/\!\!\equiv\!\!|=T_{K_V}(1,0)=(n-1)!$ total toric orders on $V$. Each one is indexed by a cyclic equivalence class of permutations
\[
  [\w]=[w_1\cdots w_n]:=\left\{w_1w_2\cdots w_{n-1}w_n,\quad w_2\cdots w_{n-1}w_nw_1,\quad\dots\quad ,w_nw_1w_2\cdots w_{n-1}\right\}.
  \]
  Geometrically, these correspond to the $(n-1)!$ toric chambers of $\AAA_{\tor}(K_V)$.

%%-------------------------------------------------
\subsection{Toric heaps}
%%-------------------------------------------------

Before we formalize cyclic reducibility, it is worth pausing to return to our familiar example for guiding intuition. Recall (see Definition~\ref{defn:toric-heap}) that a toric heap has a labeling map $\tau\colon T\to\Gamma$ such that the inverse image of every vertex and every edge is a toric chain. Moreover, $T$ must be minimal with respect to these chains. 

\begin{runningcont}
Recall that the element $w=s_3s_1s_2s_1s_2=s_3s_2s_1s_2s_1$ in $W(B_2)$ has two distinct heaps, shown in Figure~\ref{fig:b2}. Note that in each heap, the edge chain $\phi^{-1}(\{1,2\})$ forms a length-$4$ directed path in the Hasse diagram, and the size-$3$ edge chain $\phi^{-1}(\{2,3\})$ lies on a length-$4$ directed path. In the toric heap, these directed paths become toric directed paths, and so the toric Hasse diagram gets two additional edges. This is shown in Figure~\ref{fig:b2-toric-heap}; the curved edges are these two additional ones.
\end{runningcont}

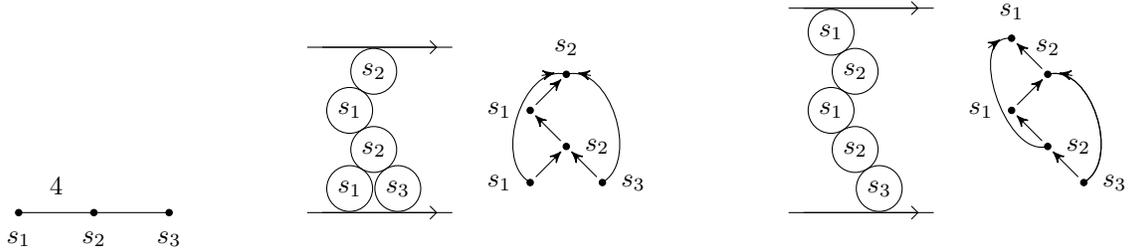
\begin{figure}[!ht]
  \begin{tikzpicture}[shorten >= 1pt, shorten <= 1pt]
    \begin{scope}[scale=.8,>=stealth',shift={(-1.5,-.5)}]
      \draw[fill=black] {(0,0) circle (1.5pt) node[label=below:$s_1$]{}};
      \draw[fill=black] {(1.25,0) circle (1.5pt) node[label=below:$s_2$]{}};
      \draw[fill=black] {(2.5,0) circle (1.5pt) node[label=below:$s_3$]{}};
      \draw {(0.625,0) node[label=above:$4$]{}};
      \draw {[-] (0,0)--(1.25,0) [-] }; \draw {[-] (1.25,0)--(2.5,0) [-] };
    \end{scope}
    \begin{scope}[scale=.8,>=angle 90,shift={(4,-.5)}]
      \draw[->] {(-.5,0) to (1.5,0)};
      \draw {(-.75,0) to (1.75,0)};
      \draw {(0,0.4) circle (0.15in) node[]{$s_1$}};
      \draw {(.4,1.05) circle (0.15in) node[]{$s_2$}};
      \draw {(0,1.7) circle (0.15in) node[]{$s_1$}};
      \draw {(.4,2.35) circle (0.15in) node[]{$s_2$}};
      \draw {(.8,0.4) circle (0.15in) node[]{$s_3$}};
      \draw[->] {(-.5,2.75) to (1.5,2.75)};
      \draw {(-.75,2.75) to (1.75,2.75)};
    \end{scope}
    \begin{scope}[scale=.8,>=stealth',shift={(7,0)}]
      \draw[fill=black] {(0,0) circle (1.5pt) node[label=left:$s_1$]{}};
      \draw[fill=black] {(0,1.2) circle (1.5pt) node[label=left:$s_1$]{}};
      \draw[fill=black] {(.6,0.6) circle (1.5pt) node[label=right:$s_2$]{}};
      \draw[fill=black] {(.6,1.8) circle (1.5pt) node[label=above:$s_2$]{}};
      \draw[fill=black] {(1.2,0) circle (1.5pt) node[label=right:$s_3$]{}};
      \draw[bigarrow] (0,0) -- (.6,0.6);
      \draw[bigarrow] (.6,0.6) -- (0,1.2);
      \draw[bigarrow] (0,1.2) -- (.6,1.8);
      \draw[bigarrow] (1.2,0) -- (.6,0.6);
      \draw[bigarrow] (.32,1.805) -- (.42,1.825); % Stupid hack
      \draw[bigarrow] (.88,1.805) -- (.78,1.825); % Stupid hack
      \draw (0,0) edge[out=145,in=170] (0.6,1.8);
      \draw (1.2,0) edge[out=35,in=10] (0.6,1.8);
    \end{scope}
    \begin{scope}[scale=.8,>=angle 90,shift={(12,-.5)}]
      \draw[->] {(-.5,0) to (1.5,0)};\draw {(-.75,0) to (1.75,0)};
      \draw {(.4,1.05) circle (0.15in) node[]{$s_2$}};
      \draw {(0,1.7) circle (0.15in) node[]{$s_1$}};
      \draw {(.4,2.35) circle (0.15in) node[]{$s_2$}};
      \draw {(.8,0.4) circle (0.15in) node[]{$s_3$}};
      \draw {(0,3.0) circle (0.15in) node[]{$s_1$}};
      \draw[->] {(-.5,3.4) to (1.5,3.4)};\draw {(-.75,3.4) to (1.75,3.4)};
    \end{scope}
    \begin{scope}[scale=.8,>=stealth',shift={(15,0)}]
      \draw[fill=black] {(0,1.2) circle (1.5pt) node[label=left:$s_1$]{}};
      \draw[fill=black] {(.6,.6) circle (1.5pt) node[label=right:$s_2$]{}};
      \draw[fill=black] {(.6,1.8) circle (1.5pt) node[label=above:$s_2$]{}};
      \draw[fill=black] {(1.2,0) circle (1.5pt) node[label=right:$s_3$]{}};
      \draw[fill=black] {(0,2.4) circle (1.5pt) node[label=above:$s_1$]{}};
      \draw (.6,.6) edge[out=190,in=180] (0,2.4);
       \draw (1.2,0) edge[out=35,in=10] (.6,1.8);
      \draw[bigarrow] (.6,0.6) -- (0,1.2);
      \draw[bigarrow] (0,1.2) -- (.6,1.8);
      \draw[bigarrow] (1.2,0) -- (.6,0.6);
      \draw[bigarrow] (.6,1.8) -- (0,2.4);
      \draw[bigarrow] (.32-.6,1.805+.51) -- (.42-.6,1.825+.56); % Stupid hack
      \draw[bigarrow] (.88,1.805) -- (.78,1.825); % Stupid hack
      \draw (1.2,0) edge[out=35,in=10] (0.6,1.8);
    \end{scope}
  \end{tikzpicture}
  \caption{Though the element $w=s_3s_1s_2s_1s_2=s_3s_2s_1s_2s_1$ in the Coxeter group $W(B_2)$ has two distinct heaps (see Figure~\ref{fig:b2}), one for each commutativity class, both of these give rise to the same toric heap. Intuitively, one can think of this as identifying the top with the bottom of a stack of balls, making it cylindrical. The undirected versions of the digraphs shown are the toric Hasse diagrams of the toric heap posets.}\label{fig:b2-toric-heap}
\end{figure}

Note that the two toric heaps shown in Figure~\ref{fig:b2-toric-heap} are actually the same, because the Hasse diagrams of the toric heap posets differ by a single source-to-sink conversion. Algebraically, this is because the word $s_3s_1s_2s_1s_2$ can be transformed into $s_3s_2s_1s_2s_1$ two ways: by a long braid relation $s_1s_2s_1s_2\mapsto s_2s_1s_2s_1$, \emph{or} by a sequence of short braid relations and cyclic shifts. When we formalize this, we will say that the cyclic words $[\w]=[s_3s_1s_2s_1s_2]$ and $[\u]=[s_3s_2s_1s_2s_1]$ are in the same \emph{cyclic commutativity class}. This is shown in Figure~\ref{fig:faux-cfc}.

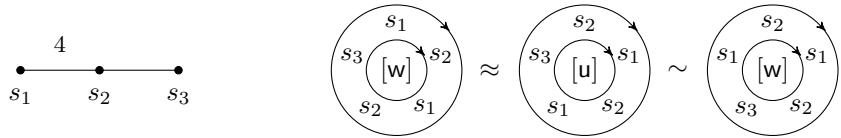
\begin{figure}[!ht]
  \begin{tikzpicture}[>=stealth', shorten >= 1pt, shorten <= 1pt]
    \begin{scope}[shift={(-5,0)},scale=.7] 
      \draw[fill=black] {(0,0) circle (2pt) node[label=below:$s_1$]{}};
      \draw[fill=black] {(1.5,0) circle (2pt) node[label=below:$s_2$]{}};
      \draw[fill=black] {(3,0) circle (2pt) node[label=below:$s_3$]{}};
      \draw {(0.75,0) node[label=above:{\small $4$}]{}};
      \draw {[-] (0,0)--(1.5,0) [-] }; \draw {[-] (1.5,0)--(3,0) [-] };
    \end{scope}
    \begin{scope}[shift={(0,0)},scale=.62]
      \draw[decoration={markings, mark=at position 0.125 with
          {\arrow{<}}}, postaction={decorate}] (0,0) circle(1.4cm);
      \draw[decoration={markings, mark=at position 0.125 with
          {\arrow{<}}}, postaction={decorate}] 
      (0cm,0cm) circle(0.65cm); 
      \draw (90:1) node{$s_1$}; 
      \draw (18:1) node{$s_2$}; 
      \draw (-52:1) node{$s_1$}; 
      \draw (-124:1) node{$s_2$};
      \draw (162:1) node{$s_3$}; 
      \node at (0,0) {$[\w]$};
      \node at (2,0) {$\approx$};
    \end{scope}
    \begin{scope}[shift={(2.5,0)},scale=.62]
      \draw[decoration={markings, mark=at position 0.125 with
          {\arrow{<}}}, postaction={decorate}] (0,0) circle(1.4cm);
      \draw[decoration={markings, mark=at position 0.125 with
          {\arrow{<}}}, postaction={decorate}] 
      (0cm,0cm) circle(0.65cm); 
      \draw (90:1) node{$s_2$}; 
      \draw (18:1) node{$s_1$}; 
      \draw (-52:1) node{$s_2$}; 
      \draw (-124:1) node{$s_1$};
      \draw (162:1) node{$s_3$}; 
      \node at (0,0) {$[\u]$};
      \node at (2,0) {$\sim$};
    \end{scope}
    \begin{scope}[shift={(5,0)},scale=.62]
      \draw[decoration={markings, mark=at position 0.125 with
          {\arrow{<}}}, postaction={decorate}] (0,0) circle(1.4cm);
      \draw[decoration={markings, mark=at position 0.125 with
          {\arrow{<}}}, postaction={decorate}] 
      (0cm,0cm) circle(0.65cm); 
      \draw (90:1) node{$s_2$}; 
      \draw (18:1) node{$s_1$}; 
      \draw (-52:1) node{$s_2$}; 
      \draw (-124:1) node{$s_3$};
      \draw (162:1) node{$s_1$}; 
      \node at (0,0) {$[\w]$};
    \end{scope}
  \end{tikzpicture}
  \caption{A non-CFC element with only one cyclic commutativity
    class. The cyclic word $[\u]$ shown in the middle differs from
    $[\w]$ via a long braid relation $s_2s_1s_2s_1\mapsto
    s_1s_2s_1s_2$, but also via a short braid relation, $s_1s_3\mapsto
    s_3s_1$. This will be formalized in
    Section~\ref{sec:cyclic-reducibility}; this type of element will
    be called \emph{torically fully commutative}
    (TFC).}\label{fig:faux-cfc}
\end{figure}

We can define structure-preserving maps between toric heaps using commutative diagrams in the same manner that we did between ordinary heaps in Definition~\ref{defn:heap-morphism}. This leads to the category $\TORHEAP$ of toric heaps and toric heap morphisms. To properly define a toric heap morphism, we need the definition of a general toric poset morphism. However, we will omit this and instead refer the reader to \cite{macauley2016morphisms}, because in this paper, the only morphisms we will use are extensions and isomorphisms. The former has been introduced and the latter is elementary: if there is a graph isomorphism $G\mapsto G'$ carrying the acyclic orientation $\omega\mapsto\omega'$, then the toric posets $T(G,[\omega])$ and $T(G',[\omega'])$ are isomorphic. Alternatively, it is straightforward to define this geometrically.

\begin{defn} \label{defn:toric-heap-morphism}
  A \emph{morphism} from one toric heap $(T,\Gamma,\tau)$ to another $(T',\Gamma',\tau')$ is a pair $(\sigma,\gamma)$, where $\sigma\colon T\to T'$ is a toric poset morphism and $\gamma\colon\Gamma\to\Gamma'$ is a graph homomorphism, satisfying $\gamma\circ\tau=\tau'\circ\sigma$:
  \[
  \xymatrix{ T \ar[rr]^\tau \ar[d]_\sigma && \Gamma \ar[d]^\gamma \\ T'\ar[rr]_{\tau'} && \Gamma'}
  \]
\end{defn}

As before, if $\sigma$ and $\gamma$ are both bijective, then the two toric heaps are \emph{isomorphic}. The concept of a toric subheap is not as natural as an ordinary subheap, because the concept of a cyclic subword of a cyclic word is not as natural as subwords of regular words. 

The concept of an extension naturally carries through from heaps to toric heaps. Recall that this is needed to properly formalize the minimality condition in Definition~\ref{defn:toric-heap}(iii). We often drop the word ``toric'' for brevity, i.e., there is no difference between an extension and a toric extension of a toric heap.

\begin{defn}
  If $(T',\Gamma',\tau')$ is the image of a morphism from a toric heap $(T,\Gamma,\tau)$, where $\sigma\colon T\to T'$ is an extension, and $\gamma\colon\Gamma\to\Gamma'$ is an edge-inclusion, then we say that it is a \emph{(toric) extension} of $(T,\Gamma,\tau)$. Moreover, it is a \emph{total (toric) extension} of toric heaps if $\sigma$ is a total toric extension of toric posets. 
\end{defn}

As we did with heaps, if we want to only consider toric heaps over a fixed graph $\Gamma$, we can define the category $\TORHEAP(\Gamma)$. However, we will generally avoid such language here, and save a thorough categorical treatment of heaps and toric heaps for a future paper. 

Now that we have laid out the fundamentals for cyclic reducibility in Coxeter groups, we will develop a framework to describe it using a cyclic analogue of a heap. The basic idea is to take our ``ball stack'' heap of pieces and wrap it in a cylinder by identifying the top with the bottom of the picture. Obviously, this idea in such a simplistic form is well suited for line graphs. In her 2014 masters thesis \cite{fox2014conjugacy}, B.~Fox proposed this idea in type $A$ Coxeter groups (though it works equally well as long as $\Gamma$ is a line graph), calling it a ``cylindrical heap.'' She acknowledged that in doing this, one loses the natural poset structure. In a 2017 paper, M.~P\'etr\'eolle defined an operation on heaps he called the ``cylindric closure'' \cite{petreolle2017characterization}. He basically turned each maximal edge chain and vertex chain into a cyclically ordered set by declaring for each vertex chain and each edge chain, the minimal element $a$ and the maximal element $z$ are related by $z\prec_ca$. This relation extends the partial order, but since antisymmetry is immediately destroyed, it is not a poset. For example, the relation between the elements in the toric directed path in Figure~\ref{fig:toric-chain} would be
\[
i_1\prec_c i_2\prec_c\cdots\prec_c i_k\prec_c i_1.
\]
In some sense, P\'etr\'eolle uses this relation $\prec_c$ as an effective ``hack'' to prove some beautiful results. Namely, he characterizes CFC elements in terms of pattern-avoidance of these cylindrical closures, and uses this to enumerate the CFC elements in all affine types. In the next section, we will formally define and develop the mathematical structure that lies behind the scenes here, and certainly elsewhere in combinatorics.

%%============================================================================
\section{Cyclic reducibility in Coxeter groups}\label{sec:cyclic-reducibility}
%%============================================================================

Throughout, let $(W,S)$ be a fixed Coxeter system with Coxeter graph $\Gamma$. As before, given words, e.g., $\u,\w,\w'$ in $S^*$, we will denote the corresponding group elements by $u,w,w'$ in $W$. At times, we can even go the other direction. Specifically, when given an element $w\in W$, we may write $\w$ to mean ``\emph{an arbitrary reduced word for $w$}.''

%%-------------------------------------------------
\subsection{Cyclic words and commutativity classes}
%%-------------------------------------------------

\begin{defn}\label{defn:cyclically-reduced}
A word $\w\in S^*$ is \emph{cyclically reduced} if every cyclic shift of it is a reduced word for some element in $W$. A group element $w\in W$ is cyclically reduced if every reduced word for $w$ is cyclically reduced. 

A word $\w\in S^*$ is \emph{torically reduced} if it remains reduced under any sequence of cyclic shifts and/or braids. A group element $w\in W$ is torically reduced if any (equivalently, every) reduced word for $w$ is torically reduced. 
\end{defn}

It is clear that torically reduced implies cyclically reduced, for both words and elements. However, the converse fails, as shown by the following example.

\begin{ex}\label{ex:3212}
The element $w=s_3s_2s_1s_2\in W(B_2)$ is cyclically reduced because it has only one reduced word, and every cyclic shift of it is reduced. However, it is not torically reduced because its cyclic shift $\u=s_2s_3s_2s_1\approx s_3s_2s_3s_1$, and this latter word is not cyclically reduced. This also shows how a word $\u\in S^*$ can be cyclically reduced despite the group element $u\in W(B_2)$ failing to be cyclically reduced. These examples are illustrated in Figure~\ref{fig:3212}.
\begin{figure}[!ht]
  \[
  \begin{tikzpicture}[scale=0.55, >=stealth', shorten >= 0pt,]
    \begin{scope}[shift={(-9,0)},scale=1.3] 
      \draw[fill=black] {(0,0) circle (2pt) node[label=below:$s_1$]{}};
      \draw[fill=black] {(1.5,0) circle (2pt) node[label=below:$s_2$]{}};
      \draw[fill=black] {(3,0) circle (2pt) node[label=below:$s_3$]{}};
      \draw {(0.75,0) node[label=above:{\small $4$}]{}};
      \draw {[-] (0,0)--(1.5,0) [-] }; \draw {[-] (1.5,0)--(3,0) [-] };
    \end{scope}
    \begin{scope}[shift={(0,0)}]
      \draw[decoration={markings, mark=at position 0.125 with
          {\arrow{<}}}, postaction={decorate}] (0,0) circle(1.5cm);
      \draw[decoration={markings, mark=at position 0.125 with
          {\arrow{<}}}, postaction={decorate}] 
      (0cm,0cm) circle(0.7cm); \draw (0,1.1)
      node{$s_3$}; \draw (1.1,0)
      node{$s_2$}; \draw (0,-1.1)
      node{$s_1$}; \draw (-1.1,0) node{$s_2$};
      \draw (0,0) node{$[\w]$}; 
      \draw (2.5, 0) node{$\approx$};
    \end{scope}
    \begin{scope}[shift={(5,0)}]
      \draw[decoration={markings, mark=at position 0.125 with
          {\arrow{<}}}, postaction={decorate}] (0,0) circle(1.5cm);
      \draw[decoration={markings, mark=at position 0.125 with
          {\arrow{<}}}, postaction={decorate}] 
      (0cm,0cm) circle(0.7cm); 
      \draw (0,1.1) node{$s_2$}; \draw (1.1,0) node{$s_3$}; 
      \draw (0,-1.1) node{$s_1$}; \draw (-1.1,0) node{$s_3$};
      \draw (2.5, 0) node{$\sim$};
    \end{scope}
    \begin{scope}[shift={(10,0)}]
      \draw[decoration={markings, mark=at position 0.125 with
          {\arrow{<}}}, postaction={decorate}] (0,0) circle(1.5cm);
      \draw[decoration={markings, mark=at position 0.125 with
          {\arrow{<}}}, postaction={decorate}] 
      (0cm,0cm) circle(0.7cm); 
      \draw (0,1.1) node{$s_2$}; \draw (1.1,0) node{$s_1$}; 
      \draw (0,-1.1) node{$s_3$}; \draw (-1.1,0) node{$s_3$};
    \end{scope}
  \end{tikzpicture}
  \]
  \caption{The element $w=s_3s_2s_1s_2$ in $W(B_2)$ is cyclically reduced but not torically reduced.}\label{fig:3212}
  \end{figure}
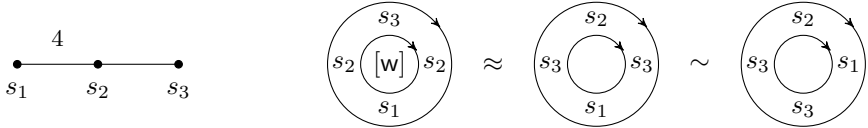
\end{ex}

At this point, we will formalize the notion of circular words, that we have been seeing in Section \ref{subsec:coxeter-elements} and in Figures \ref{fig:cfc-affine}, \ref{fig:faux-cfc}, and \ref{fig:3212}.

\begin{defn}
  Let $\w=s_{x_1}\cdots s_{x_m}$ be a word in $S^*$. The
  \emph{cyclic word} or \emph{cyclic expression} containing $\w$ is
  the equivalence class
  \[
    [\w]=[s_{x_1}\cdots s_{x_m}]:=\left\{s_{x_1}s_{x_2}\cdots s_{x_{m-1}}s_{x_m},\quad s_{x_2}\cdots s_{x_{m-1}}s_{x_m}s_{x_1},\quad\dots\quad ,s_{x_m}s_{x_1}s_{x_2}\cdots s_{x_{m-1}}\right\}.
  \]
  If $\w$ is cyclically reduced, then we say that the cyclic word $[\w]$ is reduced, and vice-versa. We say that $[\w]$ is \emph{torically reduced} if $\w$ is torically reduced. 
\end{defn}  

Let $\Cyc(W,S^*)$ denote the set of all cyclic words and let $\RRR(\Cyc(W,S^*))$ be the set of all reduced cyclic words. Write $\RRR_{\tor}(W,S^*)$ and $\RRR_{\tor}(\Cyc(W,S^*))$ for the set of torically reduced words and cyclic words, respectively. 

A word is a \emph{subword} of $[\w]$ if it appears as a (consecutive) subword of some cyclic shift of $\w$. Loosely speaking, we say that $[\w]$ and $[\u]$ differ by a braid if one can be converted into the other by replacing a subword $\<s,t\>_{m(s,t)}$ with $\<t,s\>_{m(s,t)}$. The formal definition follows.

\begin{defn}\label{defn:cyclic-braids}
Two cyclic words $[\w]$ and $[\u]$ \emph{differ by a braid} if there is some $\w'\in[\w]$ and $\u'\in[\u]$ such that $\w'$ can be converted into $\u'$ by replacing a subword of the form $\<s,t\>_{m(s,t)}$ with $\<t,s\>_{m(s,t)}$ for some $s\neq t$. 
\end{defn}

 The equivalence relations $\sim$ and $\approx$, generated by short braids and all braids, respectively, each define natural equivalence classes on the set $\RRR(W,S^*)$ of reduced words. Elements of $\RRR(W,S^*)/\!\!\sim$ are commutativity classes, and by Matsumoto's theorem, elements of $\RRR(W,S^*)/\!\!\approx$ are in 1--1 correspondence with the group elements of $W$. Since we are looking for cyclic analogues of classical results, it is natural to look at the equivalence classes of the set $\RRR_{\tor}(\Cyc(W,S^*))$ of torically reduced words defined by these relations. 

\begin{defn}\label{defn:cyclic-commutativity-class}
  Let $\w$ be torically reduced. The \emph{cyclic commutativity classes} of $[\w]$ and of $\w$ are defined as 
  \[
  \CCC_{\tor}([\w])=\{[\u]:[\u]\sim[\w]\},\qquad\CCC_{\tor}(\w)=\{\u:[\u]\in\CCC_{\tor}([\w])\}.
  \]
\end{defn}

\begin{defn}\label{defn:toric-equivalence-class}
  Let $\w$ be torically reduced. The \emph{toric equivalence classes} containing $[\w]$ and $\w$ are
 \[ 
 \RRR_{\tor}([\w])=\{[\u]:[\u]\approx[\w]\},\qquad
 \RRR_{\tor}(\w)=\{\u:[\u]\in\RRR_{\tor}([\w])\}.
 \]
 The toric equivalence class of a torically reduced element $w\in W$,
 denoted $[w]$, is defined as
  \[
    [w]=\{u\in W : \u\in\RRR_{\tor}(\w)\}.
  \]
  We write $[u]\approx[w]$ whenever $u\in[w]$. 
\end{defn}

By Matsumoto's theorem, it is well-founded to define $\RRR_{\tor}([w]):=\RRR_{\tor}([\w])$ and $\RRR_{\tor}(w):=\RRR_{\tor}(\w)$. Since $\sim$ is coarser than $\approx$, the set $\RRR_{\tor}([w])$ consists of \emph{at least} the cyclic commutativity class of $[\w]$, and potentially more. In other words, for every torically reduced $w\in W$, we have a unique decomposition
\begin{equation}\label{eqn:cyclic-decomposition}
\RRR_{\tor}([w]) = \mathcal{C}_{\tor}([\w_1])\cup\dots\cup\mathcal{C}_{\tor}([\w_N])
\end{equation}
of its torically equivalent cyclic words into cyclic commutativity classes, where each $[\w_i]\approx[\w]$, and each $\mathcal{C}_{\tor}([\w_i])$ is a union of cyclic words. We also have a similar decomposition of torically reduced words,
\begin{equation}\label{eqn:cyclic-decomposition2}
\RRR_{\tor}(w) = \mathcal{C}_{\tor}(\w_1)\cup\dots\cup\mathcal{C}_{\tor}(\w_N).
\end{equation}

It is clear that $[w]$ is contained in the conjugacy class $\cl_W(w)$. It is an open question as to when $\cl_W(w)$ contains torically reduced elements not in $[w]$. This will be discussed more in Section~\ref{sec:conjugacy}.

\begin{ex}\label{ex:3eq-classes}
The element $w=s_2s_0s_1s_0$ is torically reduced in the affine
Coxeter group $W(\wt{A}_2)$, and $\RRR_{\tor}([w])$ contains three
cyclic commutativity classes, each containing one cyclic word: $\mathcal{C}_{\tor}([s_2s_0s_1s_0])=\{[s_2s_0s_1s_0]\}$, $\mathcal{C}_{\tor}([s_2s_1s_0s_1])=\{[s_2s_1s_0s_1]\}$, and $\mathcal{C}_{\tor}([s_1s_2s_0s_2])=\{[s_1s_2s_0s_2]\}$, shown below.
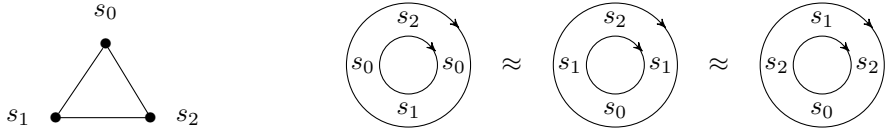
\begin{center}
  \begin{tikzpicture}[>=stealth', shorten >= -5pt, shorten <= -5pt]
    \begin{scope}[scale=.7,shift={(-6.5,-1)}]
     \node[label=left:$s_1$] (s1) at (0-.2,0) {$\tiny\bullet$};
     \node[label=right:$s_2$] (s2) at (1.4+.2,0) {$\tiny\bullet$};
     \node[label=above:$s_0$] (s0) at (.75,1.4) {$\tiny\bullet$};
     \draw (s0) -- (s1); \draw (s1) to (s2); \draw (s0) to (s2);
    \end{scope}
    \begin{scope}[scale=.55,shift={(0,0)}]
      \draw[decoration={markings, mark=at position 0.125 with
          {\arrow{<}}}, postaction={decorate}] (0,0) circle(1.5cm);
      \draw[decoration={markings, mark=at position 0.125 with
          {\arrow{<}}}, postaction={decorate}] 
      (0cm,0cm) circle(0.7cm); \draw (0,1.1)
      node{$s_2$}; \draw (1.1,0)
      node{$s_0$}; \draw (0,-1.1)
      node{$s_1$}; \draw (-1.1,0) node{$s_0$};
      \draw (2.5, 0) node{$\approx$};
    \end{scope}
    \begin{scope}[scale=.55,shift={(5,0)}]
      \draw[decoration={markings, mark=at position 0.125 with
          {\arrow{<}}}, postaction={decorate}] (0,0) circle(1.5cm);
      \draw[decoration={markings, mark=at position 0.125 with
          {\arrow{<}}}, postaction={decorate}] 
      (0cm,0cm) circle(0.7cm); \draw (0,1.1)
      node{$s_2$}; \draw (1.1,0)
      node{$s_1$}; \draw (0,-1.1)
      node{$s_0$}; \draw (-1.1,0) node{$s_1$};
      \draw (2.5, 0) node{$\approx$};
    \end{scope}
    \begin{scope}[scale=.55,shift={(10,0)}]
      \draw[decoration={markings, mark=at position 0.125 with
          {\arrow{<}}}, postaction={decorate}] (0,0) circle(1.5cm);
      \draw[decoration={markings, mark=at position 0.125 with
          {\arrow{<}}}, postaction={decorate}] 
      (0cm,0cm) circle(0.7cm); \draw (0,1.1)
      node{$s_1$}; \draw (1.1,0)
      node{$s_2$}; \draw (0,-1.1)
      node{$s_0$}; \draw (-1.1,0) node{$s_2$};
    \end{scope}
  \end{tikzpicture}
\end{center}
It is clear that all three cyclic words above lie in different cyclic commutativity classes, because they have a different multiset of generators that appear in them. Additionally, each cyclic commutativity class contains only one cyclic word because none of the generators pairwise commute. The set $\RRR_{\tor}(w)$ contains $12$ torically reduced words: four from each of the three cyclic commutativity classes. The set $[w]$ consists of the $6$ distinct group elements in $W(\widetilde{A}_2)$ defined by these $12$ words -- each one has exactly two reduced expressions for it.
\end{ex}

An interesting difference between reducibility and cyclic reducibility lies in the fact that in the cyclic case, the presence of long braid relations does not necessarily imply multiple (cyclic) commutativity classes. Specifically, $w\in W$ is FC if and only if $\RRR(w)$ contains only one commutativity class. Our running example shows how this fails in the cyclic case.  

\begin{runningcont}
  The element $w=s_3s_1s_2s_1s_2$ in $W(B_2)$ (see Figure~\ref{fig:b2-toric-heap}) is not CFC. However, $\RRR_{\tor}([w])$ contains two cyclic words, $[s_1s_2s_1s_2s_3]$ and $[s_2s_1s_2s_1s_3]$, but only one cyclic commutativity class, as shown in Figure~\ref{fig:faux-cfc}. The set $\RRR_{\tor}(w)$ consists of the ten reduced words that lie in one of these two equivalence classes. These reduced words describe four distinct group elements: 
\[
\begin{array}{ll}
s_1s_2s_1s_2s_3\approx s_2s_1s_2s_1s_3\sim s_2s_1s_2s_3s_1, \hspace{10mm} & s_1s_2s_3s_1s_2\sim s_1s_2s_1s_3s_2, \vspace{1mm} \\
s_1s_3s_2s_1s_2\sim s_3s_1s_2s_1s_2\approx s_3s_2s_1s_2s_1, & s_2s_1s_3s_2s_1\sim s_2s_3s_1s_2s_1,
\end{array}
\]
and these comprise the set $[w]$ of torically equivalent elements to $w$. 
\end{runningcont}

\begin{defn}\label{defn:faux-cfc}
  A torically reduced element with a unique cyclic commutativity class is called \emph{torically fully commutative} (TFC). If it is additionally not CFC, then we call it \emph{faux CFC}.
\end{defn}

The element $w=s_3s_1s_2s_1s_2$ in $W(B_2)$ from Running
Example~\ref{run:b2} is faux CFC. We will return to the TFC and faux
CFC elements in Section~\ref{sec:faux-cfc}.

%%-------------------------------------------------
\subsection{Toric heaps in Coxeter groups}
%%-------------------------------------------------

Just like how a word $\w\in S^*$ gives rise to a canonical heap (Definition \ref{defn:heap-coxeter}), it also gives rise to a canonical toric heap. With ordinary heaps, we are mainly only concerned with reduced words. Similarly, with toric heaps, we only need to focus on torically reduced words. 

Though many ordinary poset features have natural toric analogues, there is one crucial property that fails to carry over, and as a result, the toric heap from a word $\w\in S^*$ must be carefully defined. Every finite poset is completely determined by knowing (i) its set (a simplicial complex) of chains, and (ii) the total order within each chain. Among other things, this gives rise to the \emph{order complex} of a poset, which is fundamental to poset topology \cite{wachs2007poset}. Unfortunately, there does not appear to be a toric analogue of this property, and hence, no sort of toric order complex either. To see why, consider the toric posets $T=T(C_5,[\omega])$ and $T'=T(C_5,[\omega'])$, where $C_5$ is a 5-cycle, and the acyclic orientations $\omega$ and $\omega'$ are shown below:
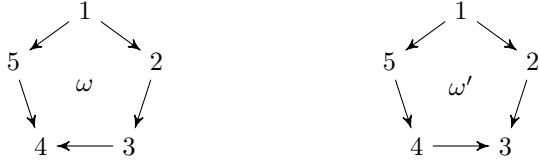
\[
  \begin{tikzpicture}[scale=1, >=stealth'] 
    \begin{scope}
      \node (1) at (90:1) {$1$}; 
      \node (2) at (18:1) {$2$}; 
      \node (3) at (-54:1) {$3$};
      \node (4) at (-126:1) {$4$}; 
      \node (5) at (162:1) {$5$};
      \draw[ba3] (1) to (2); \draw[ba3] (2) to (3); \draw[ba3] (3) to (4);
      \draw[ba3] (5) to (4); \draw[ba3] (1) to (5);
      \draw (0,0) node{$\omega$};
    \end{scope}
    \begin{scope}[shift={(5,0)}]
      \node (1) at (90:1) {$1$}; 
      \node (2) at (18:1) {$2$}; 
      \node (3) at (-54:1) {$3$};
      \node (4) at (-126:1) {$4$}; 
      \node (5) at (162:1) {$5$};
      \draw[ba3] (1) to (2); \draw[ba3] (2) to (3); 
      \draw[ba3] (4) to (3); \draw[ba3] (5) to (4); \draw[ba3] (1) to (5);
      \draw (0,0) node{$\omega'$};
    \end{scope}
  \end{tikzpicture}
  \]
  It is elementary to see that source-to-sink conversions preserve the difference between the number of clockwise edges and counterclockwise edges. Since this quantity is $3-2=1$ in $\omega$ and $2-3=-1$ in $\omega'$, then $\omega\not\equiv\omega'$, and hence $T$ and $T'$ are distinct toric posets. However, both of them have the same set of nonempty toric chains: five of size $1$ (the vertices), and five of size $2$ (the edges). Moreover, each size-$2$ toric chain trivially has the same cyclic order. Thus, to define a toric poset over a graph, it is \emph{not} enough to just specify the toric chains and the cyclic order within each one. 

The easiest way to resolve this is to simply define the toric heap poset directly from an acyclic orientation. Recall how the poset $P_\w=P(G_\w,\omega_\w)$ naturally arises from a word $\w\in S^*$ in Definition~\ref{defn:w-graph}. This also defines a natural toric poset, denoted $T_\w=T(G_\w,[\omega_\w])$. Compare the following definition to that of the heap $\HHH(\w)$ from Definition~\ref{defn:heap-coxeter}.

\begin{defn}\label{defn:toric-heap-coxeter}
  Fix a Coxeter system $(W,S)$, and let $\w=s_{x_1}\cdots s_{x_m}$ be a word in $S^*$. Define the labeling map 
  \[
  \tau_\w\colon T_\w\longto\Gamma\,,\qquad \tau_\w(i)=s_{x_i}.
  \]
  The triple $\TTT(\w):=(T_\w,\Gamma,\tau_\w)$ is called the \emph{toric heap of $\w$}.
\end{defn}

Given a word $\w\in S^*$, the labeling maps of the heap poset $P_\w$
and toric heap poset $T_\w$ are described by the following commutative diagram, 
where $\iota \colon P_\w \to T_\w$ is the identity map on the underlying set $P_\w=T_\w=[m]$.
\begin{equation}\label{eq:heap-diagram}
  \begin{gathered}
  \xymatrix{P_\w\ar[rr]^{\phi_\w}\ar[dr]_\iota && \Gamma \\ &
    T_\w\ar@{-->}[ur]_{\tau_\w} & }
  \end{gathered}
\end{equation}
Since toric directed paths are directed paths, $\iota$ pulls back toric chains in $T_\w$ to chains in $P_\w$. However, the converse need not hold. For example, consider the word $\w=s_1s_2s_3\in S^*$ in $W(A_3)$. The heap poset $P_\w$ is totally ordered. However, there is no size-$3$ toric chain in $T_\w$. 

Consider two torically reduced words in the same commutativity class, $\w\sim\w'$. By Proposition~\ref{prop:isomorphic-heaps}, the heaps $\HHH(\w)$ and $\HHH(\w')$ are isomorphic. Let $f\colon P_\w\to P_{\w'}$ be the poset isomorphism. The toric heaps $\TTT(\w)=(T_\w,\Gamma,\tau_\w)$ and $\TTT(\w')=(T_{\w'},\Gamma,\tau_{\w'})$ are related by the following commutative diagram
  \begin{equation}\label{eqn:commutative-diagram}
    \begin{gathered}
  \xymatrix{ P_\w \ar[rr]^\iota \ar[d]_{f} && T_\w \ar[rr]^{\tau_\w} \ar[d]^g && \Gamma \ar[d]^{id} \\
  				P_{\w'} \ar[rr]^{\iota'} && T_{\w'} \ar[rr]^{\tau_{\w'}} && \Gamma}
  \end{gathered}
  \end{equation}
where $g:=\iota'\circ f\circ\iota^{-1}$ is a bijection, and so $(g,id)$ is a toric poset isomorphism. A similar argument shows that if $\w$ and $\u$ differ by a single cyclic shift, then $\TTT(\w)$ and $\TTT(\u)$ are isomorphic. Specifically, define $g:T_\w\to T_{\u}$ by $g(i)=i+1\pmod{m}$, and consider only the right-hand side of the commutative diagram in Eq.~\eqref{eqn:commutative-diagram}. Since toric heaps are preserved by cyclic shifts and short braid relations, then torically equivalent words $\w\sim\u$ will yield isomorphic toric heaps. Thus, as with heaps, we will henceforth only consider toric heaps up to isomorphism.

Recall that there is a bijection between the linear extensions of a heap $\HHH(\w)$ and words in the commutativity class $\CCC(\w)$. The cyclic analogue of this holds as well. Since a total toric extension is a toric poset that is a total toric order, we usually denote it with a cyclic word, e.g., $[\u]$, instead of as a toric heap $([\u],K_n,\tau_\u)$ over the complete graph. Clearly, $[\w]$ is one total toric extension of $\TTT(\w)$. Given $\w\in S^*$, define
\begin{equation}\label{eqn:L_tor}
\LLL_{\tor}(\TTT(\w))=\{[\u]:[\u] \text{ is a total toric extension of }
\TTT(\w)\}.
\end{equation}
\begin{runningcont}
The toric heap $\TTT(\w)$ of the word $\w=s_3s_1s_2s_1s_2$ in $W(B_2)$, shown in Figures~\ref{fig:b2-toric-heap} and \ref{fig:faux-cfc}, has two total toric extensions which are both in the same cyclic commutativity class:
\[
\LLL_{\tor}(\TTT(\w)) = \{[s_3s_1s_2s_1s_2], [s_1s_3s_2s_1s_2]\}=\CCC_{\tor}([\w]).
\]
\end{runningcont}

As with the case of ordinary reducibility, there is a bijection between total toric extensions and cyclic commutativity classes. 

\begin{thm} \label{thm:cyclic-comm-class}
Let $\w$ be a torically reduced word in a Coxeter system $(W,S)$. Then
\[
\LLL_{\tor}(\TTT(\w)) = \CCC_{\tor}([\w]).
\]
\end{thm}

\begin{proof}
  Suppose $|S|=n$, and pick $[\u]$ from $\LLL_{\tor}(\TTT(\w))$, which defines a toric heap $([\u],K_n,\tau_\u)$ over the complete graph. Let $(\sigma,\gamma)$ be the extension, where $\gamma\colon\Gamma\to K_n$ is the edge-inclusion map.
  \[
  \xymatrix{ T_\w\ar[rr]^{\tau_\w} \ar[d]_\sigma && \Gamma\ar[d]^\gamma \\ [\u]\ar[rr]_{\tau_\u} && K_n}
  \]
  Since $[\w]$ and $[\u]$ are total toric extensions of $T_{\w}$, then $\w$ and $\u$ differ by a sequence of cyclic shifts and short braid relations. Therefore, $[\u]\in\CCC_{\tor}([\w])$. 

Conversely, let $[\u]\in\CCC_{\tor}([\w])$, which means there is a sequence of cyclic shifts and short braid relations that carries $\w$ to $\u$. By Corollary~5.2 in \cite{develin2016toric}, every (closed) toric chamber $\overline{{c(T_\w)}}$ is the union of closed chambers $\bar{c}_{[\w']}$ corresponding to the total toric extensions of $T_\w$. This means that $[\w]$ and $[\u]$ are both total toric extensions of $T_\w$. 
\end{proof}

\begin{ex}
  Let us revisit the Coxeter elements in the affine Coxeter group $W=W(\widetilde{A}_3)$ from Example~\ref{ex:C_4}. As before, let $\c_1=s_1s_2s_3s_4$, $\c_2=s_1s_3s_2s_4$, and $\c_3=s_1s_4s_3s_2$; see Eq.~\eqref{eq:3coxeter-elts}. Both $[\c_1]$ and $[\c_3]$ are the only cyclic words in their cyclic commutativity class, and so
\[
\RRR_{\tor}([\c_i])=\CCC_{\tor}([\c_i])=\LLL_{\tor}(\TTT(\c_i))=\{[\c_i]\}\qquad\text{for}\;\; i=1,3. 
\]
Additionally, for both $i=1,3$, the set $\RRR_{\tor}(\c_i)$ contains four reduced words -- the cyclic shifts of $\c_i$. 

In contrast, $\RRR_{\tor}([\c_2])$ contains four cyclic words that comprise a single cyclic commutativity class:
\[
\RRR_{\tor}([\c_2])=\{[s_1s_3s_2s_4],[s_1s_3s_4s_2],[s_3s_1s_2s_4],[s_3s_1s_4s_2]\}=\CCC_{\tor}([\c_2])=\LLL_{\tor}(\TTT(\c_2)).
\]
These were shown in Eq.~\eqref{eq:4words}. Finally, $\RRR_{\tor}(\c_2)$ contains 16 reduced words, four from each cyclic word above. These fall into six distinct group elements, which appeared in Eq.~\eqref{eq:16elements}. 
\end{ex}

%%==========================================================================
\section{Torically fully commutative (TFC) elements}\label{sec:faux-cfc}
%%==========================================================================

We begin this section with a reminder of a notational convention that we will use frequently: if we have a group element $w\in W$ and write $\w$, we are referring to an arbitrary fixed reduced word for $w$.

  The fully commutative (FC) elements have been well-studied; see \cite{stembridge1996fully, stembridge1998enumeration, biagioli2015fully}. It is clear that we have inclusions
\[
\begin{tikzpicture}
  \node at (0,0) {$\C(W,S)\subseteq\CFC(W,S)$};
  \node at (2,.35) {\rotatebox{30}{$\begin{array}{c}\subseteq\end{array}$}};
  \node at (2,-.35) {\rotatebox{-30}{$\begin{array}{c}\subseteq\end{array}$}};
  \node[anchor=west] at (2.2,.65) {$\FC(W,S)$}; 
  \node[anchor=west] at (2.2,-.65) {$\TFC(W,S)$};
  \end{tikzpicture}
\]
The FC elements are those that have a unique heap, i.e., for any reduced word $\u\in\RRR(w)$, we have $\HHH(w):=\HHH(\w)=\HHH(\u)$. In this section, we will establish that the torically fully commutative (TFC) elements are those that have a unique toric heap: if $\u\in\RRR_{\tor}(w)$, then $\TTT(w):=\TTT(\w)=\TTT(\u)$. After that, we will study basic properties of TFC and the faux CFC elements -- those that are TFC but not CFC. 

Say that a subset of $\Cyc(W,S^*)$ is \emph{torically order-theoretic} if it is the set of total toric extensions of a toric heap, i.e., if it can be expressed as $\LLL_{\tor}(\TTT(\u))$ for some torically reduced word $\u$. We will begin with a simple lemma; note that the converse of it is trivial. 

\begin{lem}\label{lem:torically-order-theoretic}
  If $\RRR_{\tor}([w])$ is torically order-theoretic, then $\RRR_{\tor}([w])=\LLL_{\tor}(\TTT(\w))$. 
\end{lem}

\begin{proof}
  Write $\RRR_{\tor}([w])=\CCC_{\tor}([\w_1])\cup\cdots\cup\CCC_{\tor}([\w_N])$, its unique decomposition as a disjoint union of cyclic commutativity classes, as in Eq.~\eqref{eqn:cyclic-decomposition}. Since it is torically order-theoretic, $\RRR_{\tor}([w])=\LLL_{\tor}(\TTT(\u))$ for some $\u\in\RRR_{\tor}(W,S^*)$. The proof will follow once we establish the following:
\begin{equation}\label{eqn:torically-order-theoretic}
  \CCC_{\tor}([\w])\subseteq\RRR_{\tor}([w])=\LLL_{\tor}(\TTT(\u))=\CCC_{\tor}([\u])=\CCC_{\tor}([\w])=\LLL_{\tor}(\TTT(\w)).
\end{equation}
The subset containment is immediate, and we have already seen the first equality. The second and fourth equalities are due to Theorem~\ref{thm:cyclic-comm-class}, and so now we can deduce that $\CCC_{\tor}([\w])\subseteq\CCC_{\tor}([\u])$. However, since these are non-disjoint equivalence classes, they must be the same, which confirms that the containment in Eq.~\eqref{eqn:torically-order-theoretic} is actually an equality and completes the proof of the lemma. 
\end{proof}

Though the CFC elements are a natural cyclic analogue of the FC elements, one must go up to the TFC elements to get the cyclic version of Proposition~\ref{prop:order-theoretic}.

\begin{thm}\label{thm:torically-order-theoretic}
  For a torically reduced element $w\in W$, the following are equivalent:
\begin{enumerate}[(i)]
\item $w$ is TFC.
\item $\RRR_{\tor}([w])$ is torically order-theoretic.
\item $\RRR_{\tor}([w])=\LLL_{\tor}(\TTT(\u))$ for some (equivalently, every) $\u \in \RRR_{\tor}(w)$.
\end{enumerate}
\end{thm}

\begin{proof}
The implication (ii)$\Rightarrow$(i) was established in the proof of Lemma~\ref{lem:torically-order-theoretic}, where it was shown that $\RRR_{\tor}([w])=\CCC_{\tor}([\w])$. Also simple is (iii)$\Rightarrow$(ii), which is immediate from the definition.

For (i)$\Rightarrow$(iii): If $w$ is TFC, then it has only one cyclic commutativity class, $\RRR_{\tor}([w])=\CCC_{\tor}([\w])=\LLL_{\tor}(\TTT(\w))$. The first equality is from Eq.~\eqref{eqn:cyclic-decomposition} and the second is by Theorem~\ref{thm:cyclic-comm-class}. This establishes the ``for some'' part of (iii); it suffices to prove the stronger ``for every'' part. Let $\u\in\RRR_{\tor}(w)$, which means that $[\u]\approx[\w]$, and $u$ is TFC because $w$ is. The following chain of equalities will complete the theorem:
\[
\RRR_{\tor}([w])=\RRR_{\tor}([\w])=\RRR_{\tor}([\u])=\RRR_{\tor}([u])=\LLL_{\tor}(\TTT(\u)).
\]
Specifically, the first and third are by definition, and the second is because $[\u]\approx[\w]$. The fourth equality follows from the proof of the ``for some'' implication above, applied to the TFC element $u$.
\end{proof}

Since every toric poset is completely determined by its set of total toric extensions \cite[Corollary 5.2]{develin2016toric}, Theorem~\ref{thm:torically-order-theoretic} implies that the TFC elements are those that have a unique toric heap. An open question is to classify all TFC and faux CFC elements in an arbitrary Coxeter group. In this section, we will begin with some examples, and then give some necessary and sufficient conditions for an element to be TFC or faux CFC. Throughout, $\w$ will be a torically reduced word. %in $\RRR_{\tor}(W,S^*)$.

\begin{ex} \label{ex:faux-cfc}
The following elements are all faux CFC:
\begin{enumerate}[(i)]
\item The element $w=s_3s_1s_2s_1s_2$ in $W(B_2)$ from Running Example~\ref{run:b2}.
\item The element $w=s_0s_1s_0s_1s_2s_3s_2s_3$ in the affine group $W(\widetilde{C}_3)$, which has Coxeter graph 
  
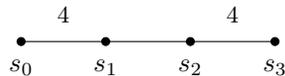
\[
  \begin{tikzpicture}[scale=.75]
    \begin{scope}[shift={(-7,0)},>=stealth']
      \draw[fill=black] {(0,0) circle (2pt) node[label=below:$s_0$]{}};
      \draw[fill=black] {(1.5,0) circle (2pt) node[label=below:$s_1$]{}};
      \draw[fill=black] {(3,0) circle (2pt) node[label=below:$s_2$]{}};
      \draw[fill=black] {(4.5,0) circle (2pt) node[label=below:$s_3$]{}};
      \draw {(0.75,0) node[label=above:{\small $4$}]{}};
      \draw {(3.75,0) node[label=above:{\small $4$}]{}};
      \draw {[-] (0,0)--(4.5,0) [-] };
    \end{scope}
  \end{tikzpicture}
    \]
\item The element $w=ststaba$ in the Coxeter group whose Coxeter graph is shown below.
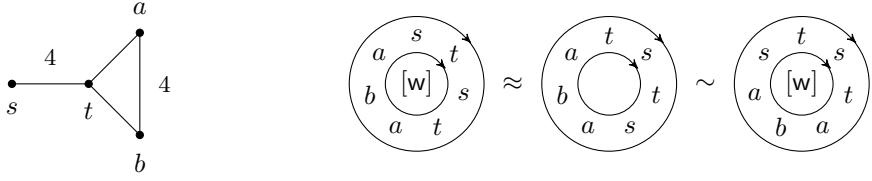
\begin{center}
  \begin{tikzpicture}[scale=.8]
    \begin{scope}[shift={(-8,0)},>=stealth',scale=.85]
      \draw[fill=black] {(1.5,0) circle (2pt) node[label=below:$s$]{}};
      \draw[fill=black] {(3,0) circle (2pt) node[label=below:$t$]{}};
      \draw[fill=black] {(4,1) circle (2pt) node[label=above:$a$]{}};
      \draw[fill=black] {(4,-1) circle (2pt) node[label=below:$b$]{}};
      \draw {(2.25,0) node[label=above:{\small $4$}]{}};
      \draw {(4,0) node[label=right:{\small $4$}]{}};
      \draw {[-] (1.5,0)--(3,0) [-] };
      \draw {[-] (3,0)--(4,1) [-] };
      \draw {[-] (3,0)--(4,-1) [-] };
      \draw {[-] (4,1)--(4,-1) [-] };
    \end{scope}
    \begin{scope}[scale=.8,shift={(0,0)},>=stealth']
      \draw[decoration={markings, mark=at position 0.125 with
          {\arrow{<}}}, postaction={decorate}] (0,0) circle(1.4cm);
      \draw[decoration={markings, mark=at position 0.125 with
          {\arrow{<}}}, postaction={decorate}] 
      (0cm,0cm) circle(0.65cm); 
      \draw (90:1) node{$s$}; 
      \draw (90-51.4:1) node{$t$}; 
      \draw (90-102.8:1) node{$s$};
      \draw (90-154.2:1) node{$t$}; 
      \draw (90+154.2:1) node{$a$};       
      \draw (90+102.8:1) node{$b$}; 
      \draw (90+51.4:1) node{$a$}; 
      \node at (2,0) {$\approx$};
      \node at (0,0) {$[\w]$};
    \end{scope}
    \begin{scope}[scale=.8,shift={(4,0)},>=stealth']
      \draw[decoration={markings, mark=at position 0.125 with
          {\arrow{<}}}, postaction={decorate}] (0,0) circle(1.4cm);
      \draw[decoration={markings, mark=at position 0.125 with
          {\arrow{<}}}, postaction={decorate}] 
      (0cm,0cm) circle(0.65cm); 
      \draw (90:1) node{$t$}; 
      \draw (90-51.4:1) node{$s$}; 
      \draw (90-102.8:1) node{$t$};
      \draw (90-154.2:1) node{$s$}; 
      \draw (90+154.2:1) node{$a$};       
      \draw (90+102.8:1) node{$b$}; 
      \draw (90+51.4:1) node{$a$}; 
      \node at (2,0) {$\sim$};
    \end{scope}
    \begin{scope}[scale=.8,shift={(8,0)},>=stealth']
      \draw[decoration={markings, mark=at position 0.125 with
          {\arrow{<}}}, postaction={decorate}] (0,0) circle(1.4cm);
      \draw[decoration={markings, mark=at position 0.125 with
          {\arrow{<}}}, postaction={decorate}] 
      (0cm,0cm) circle(0.65cm); 
      \draw (90:1) node{$t$}; 
      \draw (90-51.4:1) node{$s$}; 
      \draw (90-102.8:1) node{$t$};
      \draw (90-154.2:1) node{$a$}; 
      \draw (90+154.2:1) node{$b$};       
      \draw (90+102.8:1) node{$a$}; 
      \draw (90+51.4:1) node{$s$}; 
      \node at (0,0) {$[\w]$};
    \end{scope}
  \end{tikzpicture}
\end{center}
\end{enumerate}
\end{ex}

It is obvious that a faux CFC element cannot contain a long braid relation of odd length, because applying a single such relation changes the number of individual generators in the word. 

\begin{prop}\label{prop:faux-cfc}
  If $w$ is faux CFC and $m(s,t)\geq 3$ is odd, then none of the torically reduced words $\u\in\RRR_{\tor}(w)$ contain $\<s,t\>_{m(s,t)}$ as a subword. $\hfill\Box$
\end{prop}

The previous result is a necessary condition for $w$ to be faux CFC. Next, we will provide a sufficient condition. Say that $s\in S$ is an \emph{endpoint} of the Coxeter graph $\Gamma$ if it has degree $1$, and is even (respectively, odd) if $m(s,t)$ is even (respectively, odd) for the unique $t$ for which $m(s,t)>2$. In this case, we call $\{s,t\}$ a \emph{spoke} of $\Gamma$, and can speak of spokes being even or odd. 

\begin{prop} \label{prop:build-faux-CFC}
Let $\Gamma$ be a Coxeter graph with an even endpoint $s$ with adjacent vertex $t$.
Suppose $\w=\<s,t\>_{m(s,t)}\u$ is a reduced word for some element $w \in W$
where the subword $\u$ satisfies the following two properties:
\begin{itemize}
	\item $\u$ contains no generators $s$ and $t$, and 
	\item $\u$ is a reduced word for some CFC element in $W$.
\end{itemize}
Then $w$ is TFC. 
\end{prop}

\begin{proof}
Since $\u$ is a reduced word for some CFC element containing no $s$ nor $t$,
then $\w=\<s,t\>_{m(s,t)}\u$ is torically reduced, and the only long braid relation that can be applied to $[\w]$ is $\<s,t\>_{m(s,t)}=\<t,s\>_{m(s,t)}$.
Thus, $w$ is torically reduced but not CFC.
Since $s$ commutes with every generator in $\u$,
\[
  [\w]=[ \underbrace{st \cdots st}_{m(s,t)} \u ]  
  \approx [ \underbrace{ts \cdots ts}_{m(s,t)} \u ]  
  \sim [ \underbrace{ts \cdots ts}_{m(s,t)-2} t \u s ]  
  = [ s\underbrace{ts \cdots ts}_{m(s,t)-2} t \u  ]  = [\w].
\]
That is, $\RRR_{\tor}([\w])$ has only one equivalence (cyclic commutativity) class, and so $w$ is TFC. 
\end{proof}

We conjecture that faux CFC elements (with full support) only occur 
when the Coxeter graph has at least one even endpoint. It would be desirable to understand how faux CFC elements can be ``reduced'' down to CFC elements. In our prior examples of faux CFC elements of the form $\<s,t\>_{m(s,t)}\u$, removing $st$ from the word yields an element that can be:
\begin{enumerate}[(i)]
	\item CFC, as in Example~\ref{ex:faux-cfc}(i), or
	\item faux CFC, as in Example~\ref{ex:faux-cfc}(ii), or 
	\item not TFC, as in Example~\ref{ex:faux-cfc}(iii).
\end{enumerate}
Whether the new word $\<s,t\>_{{m(s,t)}-2}\u$ is CFC, faux CFC, or not TFC, 
depends on $\u$. For the three cases above, $\u$ is (i) CFC, (ii) faux CFC, (iii) not cyclically reduced. We end this section with a conjecture.

\begin{conj}\label{conj:faux-cfc}
If $\w=\<s,t\>_{m(s,t)}\u$ is faux CFC and $\u$ torically reduced, then $\<s,t\>_{{m(s,t)}-2}\u$ is TFC. 
\end{conj}

%%==========================================================================
\section{Toric heaps and conjugacy in Coxeter groups}\label{sec:conjugacy}
%%==========================================================================

We will wrap up this paper by revisiting some of the existing results on Coxeter theory from Section~\ref{sec:coxeter-fc-cfc}, cast them in a cyclic reducibility framework, and discuss some open problems. 

\begin{thm}\label{thm:coxeter-elts}
  Let $c_1$ and $c_2$ be Coxeter elements of $W$. Then
\begin{enumerate}[(i)]
\item $c_1=c_2$ if and only if they have the same heap, and
\item $c_1$ and $c_2$ are conjugate if and only if they have the same
  toric heap.
\end{enumerate}
\end{thm}

The first statement in Theorem~\ref{thm:coxeter-elts} is an immediate consequence of Matsumoto's theorem, and the second is Theorem~\ref{thm:erikssons} using the toric heap language. Note that this fails if we drop the assumption that $w$ is a Coxeter element, because in the finite Coxeter group $W(A_2)$,
\begin{equation}\label{eq:A_2}
(s_1s_2)s_1(s_1s_2)^{-1}=s_1(s_2s_1s_2)s_1=s_1(s_1s_2s_1)s_1=s_2. 
\end{equation}
More generally, examples like this exist when distinct subsets $T,U\subseteq S$, and hence, their standard parabolic subgroups, $W_T$ and $W_U$, are conjugate in $W$. This phenomenon has been completely characterized by Deodhar~\cite{deodhar1982root}, and it only happens when $W_T$ is finite. The technical condition can be found in \cite[Theorem 3.1.3]{krammer2009conjugacy}.

The proof of the second statement in Theorem~\ref{thm:coxeter-elts} heavily relies on two properties of Coxeter elements. The first is toric equivalence, and the second is a theorem proven independently by Speyer in 2009 \cite{speyer2009powers} and by the Erikssons in 2010 \cite{eriksson2010words}: In an infinite irreducible Coxeter group, all Coxeter elements are logarithmic. It is a simple exercise to extend this to the general case of reducible Coxeter groups. Of course, the necessary and sufficient condition is that every connected component of $\Gamma$ describes an infinite group. In \cite{boothby2012cyclically}, a CFC element $w\in W$ was said to be \emph{torsion free} if every factor of the parabolic subgroup $W_{\supp(w)}$ describes an infinite group. Equivalently, none of the connected components of the induced graph $\Gamma[\supp(w)]$ are of finite-type. It has been proven (see \cite{marquis2014conjugacy,boothby2012cyclically}) that CFC elements are logarithmic if and only if they are torsion free. Marquis proved that a similar result holds more generally, but it requires a relaxation of the definition of torsion free.

\begin{defn}[\cite{marquis2014conjugacy}]
An element $w\in W$ is \emph{torsion free} if it has no reduced decomposition of the form $w=w_In_I$ for some spherical subset $I\subseteq S$, some $w_I\in W_I\setmin\{1\}$, and some $n_I$ normalizing $W_I$.
\end{defn}

To see why this is a necessary requirement to being logarithmic, suppose that $w$ is not torsion free and write $w=w_In_I$. Then for each $k\in\Z$, we can write $w^k=w_kn_I^k$ for some $w_k\in W_I$. Since $W_I$ is finite, there must be some $m>k$ such that $u:=w_k=w_m$, and so
\[
w^{m-k}=w^{-k}w^m=(n_I^{-k}u^{-1})(un_I^m)=n_I^{m-k}.
\]
Note that such a $w$ is not logarithmic because
\[
\ell(w^{m-k})=\ell(n_I^{m-k})\leq(m-k)\ell(n_I)<(m-k)(\ell(w_I)+\ell(n_I))\leq(m-k)\ell(w).
\]
Remarkably, Marquis proved that up to toric equivalence, this is the \emph{only} barrier to a torically reduced element being logarithmic. 

\begin{thm}[\cite{marquis2014conjugacy}]\label{thm:marquis}
A torically reduced $w\in W$ is logarithmic if and only if every torically equivalent $u\in[w]$ is torsion-free. 
\end{thm}

A simple example of a non torsion-free element in an infinite group can be found by modifying our Running Example from living in the group $W(B_2)$ to $W(\widetilde{C}_2)$. 

\begin{ex}\label{ex:affine-C2}
Consider the faux CFC element $w=s_0 s_1 s_0 s_1 s_2$ in $\widetilde{C}_2$, as shown below:
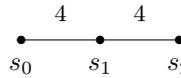
\[
  \begin{tikzpicture}[scale=.7,>=stealth', shorten >= 1pt, shorten <= 1pt]
    \begin{scope}[shift={(-5,0)}]
      \draw[fill=black] {(0,0) circle (2pt) node[label=below:$s_0$]{}};
      \draw[fill=black] {(1.5,0) circle (2pt) node[label=below:$s_1$]{}};
      \draw[fill=black] {(3,0) circle (2pt) node[label=below:$s_2$]{}};
      \draw {(0.75,0) node[label=above:{\small $4$}]{}};
      \draw {(2.25,0) node[label=above:{\small $4$}]{}};
      \draw {[-] (0,0)--(1.5,0) [-] }; \draw {[-] (1.5,0)--(3,0) [-] };
    \end{scope}
    \begin{scope}[shift={(1,1)},scale=.9]
      \node (1) at (.3,0.1) {\small $w^2$};
      \node (2) at (1,0) {\small $=$};
      \node (3) at (4.2,0) {\small $(s_0s_1s_0s_1s_2)(s_0s_1s_0s_1s_2)$};
      \node (4) at (1,-1) {\small $=$};
      \node (5) at (4.2,-1) {\small $(s_1s_0s_1s_0s_2)(s_0s_1s_0s_1s_2)$};
      \node (6) at (1,-2) {\small $=$};
      \node (7) at (3.4,-2) {\small $s_1s_0s_1s_2s_1s_0s_1s_2$.};
	\end{scope}
  \end{tikzpicture}
\]
 Clearly, $w$ is faux CFC but is not logarithmic because $\ell(w^2)<2\ell(w)$. It is not torsion-free because taking $I=\{s_0\}$, we can write $w=w_In_I=n_Iw_I$ for $w_I=s_0$ and $n_I=s_1s_0s_1s_2$.
\end{ex}

Thus far, we have seen how cyclic reducibility affects ordinary reducibility. Now, we ask how it affects conjugacy.  

Matsumoto's theorem says that if $\u$ and $\w$ are reduced words for the same group element, then they differ by braids, i.e., $\u\approx\w$. We are interested in a ``cyclic version'' of Matsumoto's theorem, which asks whether the cyclic words of torically reduced conjugate elements differ by braids, i.e., $[\u]\approx[\w]$. This is to the conjugacy problem what Matsumoto's theorem is to the word problem. 

\begin{defn} 
  A $W$-conjugacy class $C$ satisfies the \emph{cyclic version of Matsumoto's theorem} (CVMT) if any two reduced words of torically reduced elements in $C$ differ by braid relations and cyclic shifts.
\end{defn}

As we have seen in Eq.~\eqref{eq:A_2}, the CVMT trivially fails in conjugacy classes that have a minimal cyclically reduced element with support $T\subseteq S$ conjugate to some other $U\subseteq S$. We conjecture that this is the only times where it fails. In a recent paper~\cite{marquis2014conjugacy}, Marquis proved the CVMT for all elements $w\in W$ of infinite order with property ($\Cent$), which means that whenever $w$ normalizes a finite parabolic subgroup of $W$, it centralizes the subgroup. Though this condition may seem peculiar, it is quite mild, as it is satisfied by any torsion-free normal subgroup $W_0$ of finite index, which always exists in any infinite Coxeter group by Selberg's Lemma and the linearity of Coxeter groups \cite[Lemma 1]{dranishnikov1999every}. Torically reduced elements that have property ($\Cent$), or are in finite groups, are additionally strongly cyclically reduced \cite[Corollary C]{marquis2014conjugacy}, which means that $\ell(w)=\min\{\ell(vwv^{-1})\mid v\in W\}$, or equivalently, $w$ is of minimal length in its conjugacy class. As a corollary, it follows that the CVMT holds for torsion-free CFC elements. Equivalently, Theorem~\ref{thm:coxeter-elts} can be extended to this class of elements. 

\begin{conj}
The cyclic version of Matsumoto's theorem holds for a $W$-conjugacy classes, as long as the parabolic closure $W_{\supp(w)}$ of its minimal elements contains only infinite irreducible components. 
\end{conj}

Finally, it should be noted that Marquis' aforementioned results were incorrectly stated with ``cyclically reduced'' instead of ``torically reduced''. For example, the cyclically reduced element $w=s_3s_2s_1s_2$ in $W(B_2)$ from Example~\ref{ex:3212} is not strongly cyclically reduced as \cite[Corollary C]{marquis2014conjugacy} would suggest, because 
\[
(s_2s_3)^{-1}(s_3s_2s_1s_2)(s_2s_3)=s_3(s_2s_3s_2s_1)s_3=s_3(s_3s_2s_3s_1)s_3=s_2s_1.
\]
Fortunately, this flaw is easily fixable. The author defined cyclically reduced elements as we did in Definition~\ref{defn:cyclically-reduced}, but then in the proofs, used the subtly stronger concept of being torically reduced.

%%===========================================================================
\section{Concluding remarks}\label{sec:conclusions}
%%===========================================================================

The purpose of this paper is to present a framework for studying what we call ``cyclic reducibility'' in Coxeter groups, and show how this relates to ordinary problems in reducibility and conjugacy. We introduced the concept of a \emph{toric heap}, which is a labeled toric poset and a cyclic version of Viennot's classic ``heaps of pieces.'' We also formalized concepts such as cyclic words, cyclic commutativity classes, as well as cyclically and torically reduced words and elements in Coxeter groups. The notion of a toric heap morphism allowed us to formalize concepts such as labeled toric extensions, and prove that there is a bijection between (labeled) total toric extensions of the toric heap of a word, and its cyclic commutativity classes. This is a cyclic analogue of a theorem by Stembridge about labeled linear extensions of a heap and commutativity classes \cite{stembridge1996fully}. However, not all fundamental properties carried over. We saw how elements can admit long braid relations but still have a unique toric heap. This brought up a new class of elements called \emph{torically fully commutative} (TFC), which generalized the notion of the cyclically fully commutative (CFC) elements. These elements led to another theorem of Stembridge that generalized a characterization of the FC elements. Classifying elements that are TFC but not CFC remains an open question (see Conjecture~\ref{conj:faux-cfc}). Our cyclic reducibility framework casts old and new results in a more transparent light. We are currently exploring whether techniques and proofs involving cyclic reducibility and conjugacy can be simplified and streamlined using toric heaps.

  Toric heaps should be of general interest outside of Coxeter groups, both as interesting mathematical objects in their own right, and as a convenient framework when labeled cyclic poset structures arise in various applications. A common theme with toric poset research has been to explore toric analogues of features of ordinary posets. Similarly, ordinary heaps have been applied in many settings, and some of these should have natural toric analogues that can be explored. We are also working on a more thorough analysis of the categories of heaps and toric heaps. Regarding applications, it is difficult to predict where and how toric heaps might arise in the future. Naturally, this should not come as a surprise; back in 1986, Viennot surely did not anticipate his new theory of heaps springing up in Coxeter groups, Lorentzian quantum gravity \cite{viennot2014heaps}, models for polymers \cite{brak2007motzkin}, modeling with Petri nets \cite{gaubert1999modeling}, and discrete-event systems \cite{su2012synthesis}.

%%===================================================================
%%===================================================================

%%===================================================================

\begin{thebibliography}{10}

\bibitem{adin2018cyclic}
R.~M. Adin, I.~M. Gessel, V.~Reiner, and Y.~Roichman.
\newblock Cyclic quasi-symmetric functions.
\newblock In {\em $31^{\rm th}$ Int'l Conf. on Formal Power Series and
  Algebraic Combonatorics (FPSAC)}, volume 82B, pages Article \#67, 12 pp.,
  Ljubljana, 2019.

\bibitem{biagioli2015fully}
R.~Biagioli, F.~Jouhet, and P.~Nadeau.
\newblock Fully commutative elements in finite and affine {C}oxeter groups.
\newblock {\em Monatsh. Math.}, 178(1):1--37, 2015.

\bibitem{billey2001kazhdan}
S.~C. Billey and G.~S. Warrington.
\newblock Kazhdan-{L}usztig polynomials for 321-hexagon-avoiding permutations.
\newblock {\em J. Algebraic Combin.}, 13(2):111--136, 2001.

\bibitem{bjorner2005combinatorics}
A.~Bj{\"o}rner and F.~Brenti.
\newblock {\em Combinatorics of {C}oxeter Groups}.
\newblock Springer Verlag, New York, 2005.

\bibitem{boothby2012cyclically}
T.~Boothby, J.~Burkert, M.~Eichwald, D.~C. Ernst, R.~M. Green, and M.~Macauley.
\newblock On the cyclically fully commuative elements of {C}oxeter groups.
\newblock {\em J. Algebraic Combin.}, 36(1):123--148, 2012.

\bibitem{bousquet2002lattice}
M.~Bousquet-M{\'e}lou and A.~Rechnitzer.
\newblock Lattice animals and heaps of dimers.
\newblock {\em Discrete Math.}, 258(1-3):235--274, 2002.

\bibitem{bousquet1992empilements}
M.~Bousquet-M{\'e}lou and X.~G. Viennot.
\newblock Empilements de segments et q-{\'e}num{\'e}ration de polyominos
  convexes dirig{\'e}s.
\newblock {\em J. Combin. Theory Ser. A}, 60(2):196--224, 1992.

\bibitem{brak2007motzkin}
R.~Brak, G.~K. Iliev, A.~Rechnitzer, and S.~G. Whittington.
\newblock Motzkin path models of long chain polymers in slits.
\newblock {\em J. Phys. A}, 40(17):4415, 2007.

\bibitem{cartier1969problemes}
P.~Cartier and D.~Foata.
\newblock {\em Problemes combinatoires de commutation et re\'arrangements},
  volume~85 of {\em Lecture Notes in Mathematics}.
\newblock Springer Verlag, 1969.

\bibitem{deodhar1982root}
V.~V. Deodhar.
\newblock On the root system of a {C}oxeter group.
\newblock {\em Commun. Algebra}, 10(6):611--630, 1982.

\bibitem{develin2016toric}
M.~Develin, M.~Macauley, and V.~Reiner.
\newblock Toric partial orders.
\newblock {\em Trans. Amer. Math. Soc.}, 368(4):2263--2287, 2016.

\bibitem{di2010q-systems}
P.~Di~Francesco and R.~Kedem.
\newblock Q-systems, heaps, paths and cluster positivity.
\newblock {\em Commun. Math. Phys.}, 293(3):727, 2010.

\bibitem{dranishnikov1999every}
A.~N. Dranishnikov and T.~Januszkiewicz.
\newblock Every {C}oxeter group acts amenably on a compact space.
\newblock {\em Topology Proc.}, 24:135--141, 1999.

\bibitem{eriksson2009conjugacy}
H.~Eriksson and K.~Eriksson.
\newblock Conjugacy of {C}oxeter elements.
\newblock {\em Electron. J. Combin.}, 16(2):\#R4, 2009.

\bibitem{eriksson2010words}
H.~Eriksson and K.~Eriksson.
\newblock Words with intervening neighbours in infinite {C}oxeter groups are
  reduced.
\newblock {\em Electron. J. Combin.}, 17(1):\#N9, 2010.

\bibitem{fedou1995fonctions}
J.-M. F\'{e}dou.
\newblock Sur les fonctions de {B}essel.
\newblock {\em Discrete Math.}, 139(1-3):473--480, 1995.

\bibitem{fox2014conjugacy}
B.~Fox.
\newblock Conjugacy classes of cyclically fully commutative elements in
  {C}oxeter groups of type {A}.
\newblock Master's thesis, Northern Arizona University, 2014.

\bibitem{gaubert1999modeling}
S.~Gaubert and J.~Mairesse.
\newblock Modeling and analysis of timed {P}etri nets using heaps of pieces.
\newblock {\em IEEE Trans. Automat. Contr.}, 44(4):683--697, 1999.

\bibitem{green2007full}
R.~M. Green.
\newblock Full heaps and representations of affine {K}ac-{M}oody algebras.
\newblock {\em Int. Electron. J. Algebra}, 2:137--188, 2007.

\bibitem{green2010combinatorics}
R.~M. Green.
\newblock On the combinatorics of {K}ac's asymmetry function.
\newblock {\em Comment. Math. Univ. Carolin.}, 51(2):217--235, 2010.

\bibitem{green2002freely}
R.~M. Green and J.~Losonczy.
\newblock Freely braided elements in {C}oxeter groups.
\newblock {\em Ann. Comb.}, 6:337--348, 2002.

\bibitem{humphreys1990reflection}
J.~E. Humphreys.
\newblock {\em Reflection Groups and {C}oxeter Groups}.
\newblock Cambridge University Press, Cambridge, UK, 1990.

\bibitem{krammer2009conjugacy}
D.~Krammer.
\newblock The conjugacy problem for {C}oxeter groups.
\newblock {\em Groups Geom. Dyn.}, 3(1):71--171, 2009.

\bibitem{lalonde1995lyndon}
P.~Lalonde.
\newblock Lyndon heaps: an analogue of {L}yndon words in free partially
  commutative monoids.
\newblock {\em Discrete Math.}, 145(1-3):171--189, 1995.

\bibitem{macauley2016morphisms}
M.~Macauley.
\newblock Morphisms and order ideals of toric posets.
\newblock {\em Mathematics}, 4(2):31 pages, 2016.

\bibitem{marquis2014conjugacy}
T.~Marquis.
\newblock Conjugacy classes and straight elements in coxeter groups.
\newblock {\em J. Algebra}, 407:68--80, 2014.

\bibitem{matsumoto1964generateurs}
H.~Matsumoto.
\newblock G\'en\'erateurs et relations des groupes de weyl g\'en\'eralis\'es.
\newblock {\em C. R. Acad. Sci. Paris}, 258:3419--3422, 1964.

\bibitem{petreolle2016generating}
M.~P{\'e}tr{\'e}olle.
\newblock Generating series of cyclically fully commutative elements is
  rational, 2016.
\newblock ArXiv Preprint. math.CO/1612.03764.

\bibitem{petreolle2017characterization}
M.~P\'etr\'eolle.
\newblock Characterization of cyclically fully commutative elements in finite
  and affine {C}oxeter groups.
\newblock {\em Eur. J. Combin.}, 61:106--132, 2017.

\bibitem{pretzel1986reorienting}
O.~Pretzel.
\newblock On reorienting graphs by pushing down maximal vertices.
\newblock {\em Order}, 3(2):135--153, 1986.

\bibitem{speyer2009powers}
D.~E. Speyer.
\newblock Powers of {C}oxeter elements in infinite groups are reduced.
\newblock {\em Proc. Amer. Math. Soc.}, 137:1295--1302, 2009.

\bibitem{stembridge1996fully}
J.~R. Stembridge.
\newblock On the fully commutative elements of {C}oxeter groups.
\newblock {\em J. Algebraic Combin.}, 5:353--385, 1996.

\bibitem{stembridge1998enumeration}
J.~R. Stembridge.
\newblock The enumeration of fully commutative elements of {C}oxeter groups.
\newblock {\em J. Algebraic Combin.}, 7(3):291--320, 1998.

\bibitem{su2012synthesis}
R.~Su, J.~H. Van~Schuppen, and J.~E. Rooda.
\newblock The synthesis of time optimal supervisors by using heaps-of-pieces.
\newblock {\em IEEE Trans. Automat. Contr.}, 57(1):105--118, 2012.

\bibitem{viennot1986combinatoire}
G.~X. Viennot.
\newblock Heaps of pieces, {I}: {B}asic definitions and combinatorial lemmas.
\newblock In G.~Labelle and P.~Leroux, editors, {\em Combinatoire
  \'Enum\'erative}, pages 321--350. Springer-Verlag, 1986.

\bibitem{viennot2014heaps}
X.~G. Viennot.
\newblock Heaps of pieces and {2D} {L}orentzian quantum gravity.
\newblock In {\em Workshop on {2D} quantum gravity and statistical mechanics,
  Erwin Schr{\"o}dinger Institute}, 2014.

\bibitem{wachs2007poset}
M.~L. Wachs.
\newblock Poset topology: Tools and applications.
\newblock In {\em Geometric Combinatorics, IAS/Park City Mathematics Series},
  pages 497--615, 2007.

\bibitem{wildberger2003combinatorial}
N.~J. Wildberger.
\newblock A combinatorial construction for simply-laced {L}ie algebras.
\newblock {\em Adv. Appl. Math.}, 30(1-2):385--396, 2003.

\end{thebibliography}
\end{document}